\documentclass[reqno,10pt, centertags]{amsart}
\usepackage{amssymb,upref,esint,color}
\usepackage{hyperref}
\usepackage{etoolbox}

\newcommand*{\mailto}[1]{\href{mailto:#1}{\nolinkurl{#1}}}
\newcommand{\arxiv}[1]{\href{http://arxiv.org/abs/#1}{arXiv:#1}}
\makeatletter
\def\theequation{\@arabic\c@equation}

\newcommand{\diag}{\operatorname{diag}}

\newcommand{\bbN}{{\mathbb{N}}}
\newcommand{\bbR}{{\mathbb{R}}}

\newcommand{\bbZ}{{\mathbb{Z}}}
\newcommand{\bbC}{{\mathbb{C}}}

\newcommand{\cB}{{\mathcal B}}

\newcommand{\cD}{{\mathcal D}}

\newcommand{\cF}{{\mathcal F}}

\newcommand{\cH}{{\mathcal H}}

\newcommand{\cK}{{\mathcal K}}

\newcommand{\cM}{{\mathcal M}}
\newcommand{\cN}{{\mathcal N}}

\newcommand{\dott}{\,\cdot\,}

\newcommand{\no}{\nonumber}
\newcommand{\lb}{\label}
\newcommand{\f}{\frac}

\newcommand{\ol}{\overline}

\newcommand{\wti}{\widetilde}

\newcommand{\Oh}{O}

\newcommand{\tr}{\text{\rm{tr}}}

\newcommand{\ran}{\text{\rm{ran}}}
\newcommand{\ind}{\text{\rm{ind}}}
 
\newcommand{\dom}{\text{\rm{dom}}}

\newcommand{\bi}{\bibitem}
\newcommand{\sgn}{\text{\rm{sign}}}
\newcommand{\hatt}{\widehat}

\renewcommand{\ln}{\text{\rm ln}}
\renewcommand{\diag}{\text{\rm diag}}
\newcommand{\linspan}{\operatorname{lin.span}} 
\renewcommand{\dot}{\overset{\textbf{\Large.}}}

\numberwithin{equation}{section}

\newtheorem{theorem}{Theorem}[section]
\newtheorem{lemma}[theorem]{Lemma}
\newtheorem{corollary}[theorem]{Corollary}
\newtheorem{proposition}[theorem]{Proposition}
\newtheorem{hypothesis}[theorem]{Hypothesis}
\newtheorem{example}[theorem]{Example}
\theoremstyle{definition}
\newtheorem{definition}[theorem]{Definition}
\newtheorem{remark}[theorem]{Remark}

\begin{document}

\title[The Generalized Birman--Schwinger Principle]{The Generalized Birman--Schwinger Principle}

\author[J.\ Behrndt]{Jussi Behrndt}  
\address{Institut f\"ur Angewandte Mathematik, Technische Universit\"at 
Graz, Steyrergasse 30, 8010 Graz, Austria
}  
\address{Department of Mathematics, Stanford University, 450 Jane Stanford Way, Stanford CA 94305-2125, USA}
\email{\mailto{behrndt@tugraz.at}, \mailto{jbehrndt@stanford.edu}}
\urladdr{\url{http://www.math.tugraz.at/~behrndt/}}

\author[A. ter Elst]{A.\ F.\ M.\ ter Elst}
\address{Department of Mathematics, University of Auckland, 
Private Bag 92019, Auckland 1142, New Zealand}
\email{\mailto{terelst@math.auckland.ac.nz}}
%\email{terelst@math.auckland.ac.nz}
\urladdr{\href{https://www.math.auckland.ac.nz/people/ater013}{https://www.math.auckland.ac.nz/people/ater013}}
%\urladdr{https://www.math.auckland.ac.nz/people/ater013}}
%\urladdr{https://www.math.auckland.ac.nz/people/ater013}

\author[F.\ Gesztesy]{Fritz Gesztesy}
\address{Department of Mathematics, 
Baylor University, One Bear Place \#97328,
Waco, TX 76798-7328, USA}
\email{\mailto{Fritz\_Gesztesy@baylor.edu}}
%\email{Fritz$\_$Gesztesy@baylor.edu}
\urladdr{\url{http://www.baylor.edu/math/index.php?id=935340}}
%\urladdr{http://www.baylor.edu/math/index.php?id=935340}

%\dedicatory{.}

\date{\today}
%\date{, 2003.}
%\thanks{ } 
%\thanks{Appeared in {\it } .}

%%%%%%% 
\begin{abstract}
We prove a generalized Birman--Schwinger principle in the non-self-adjoint context. In particular, we provide a 
detailed discussion of geometric and algebraic multiplicities of eigenvalues of the basic operator of interest (e.g., a Schr\"odinger operator) 
and the associated Birman--Schwinger operator, and additionally offer a careful study of the associated Jordan chains of generalized eigenvectors of both operators.
In the course of our analysis we also study algebraic and geometric multiplicities of zeros of strongly analytic operator-valued 
functions and the associated Jordan chains of generalized eigenvectors. We also relate algebraic multiplicities to the notion of the index of analytic 
operator-valued functions and derive a general Weinstein--Aronszajn formula for a pair of non-self-adjoint operators. 
\end{abstract}
%%%%%%% 

\begingroup
% Hack to get rid of the amsart all-caps style in the authors.
% http://tex.stackexchange.com/questions/2820/disable-toggle-smallcaps
\makeatletter
\patchcmd{\@setauthors}{\MakeUppercase}{\sc}{}{}
\makeatother
\maketitle
\endgroup

%\newpage 

{\scriptsize{\tableofcontents}}
%\normalsize

%%%%%%% 
%%%%%%% 
\section{Introduction} \lb{s1} 
%%%%%%% 
%%%%%%% 

The Birman--Schwinger principle 
is one of the standard tools in spectral analysis of Schr\"{o}dinger operators, originating in work of Birman \cite{B66} and Schwinger \cite{Sc61}, and raised to an art by mathematical physicists in subsequent decades. In its original form, this useful technique permits one to reduce the eigenvalue problem 
for an unbounded differential operator (e.g., the Schr\"{o}dinger operator $H=-\Delta +V$ in 
$L^2(\bbR^n; d^nx)$) to an eigenvalue problem for a bounded integral operator involving a sandwiched resolvent of the unperturbed operator 
(e.g., $H_0=-\Delta$ in the case of Schr\"odinger operators), where the underlying integral kernel and its mapping properties and asymptotics 
are well-studied. Roughly speaking, in the standard Schr\"odinger operator situation, $z_0\in\bbC\backslash [0,\infty)$ is an eigenvalue 
of $H=-\Delta +V$ in $L^2(\bbR^n; d^nx)$
if and only if $-1$ is an eigenvalue of the Birman--Schwinger
operator $V_1(H_0-z_0 I_{L^2(\bbR^n; d^nx)})^{-1}V_2^*$, employing a convenient factorization $V=V_2^*V_1$.
This correspondence is very useful in various spectral problems; typical examples are explicit bounds on the number of discrete eigenvalues in essential spectral gaps, in the proof of the Lieb--Thirring inequality \cite{LT76}, the proof of the Cwikel--Lieb-Rozenblum bound, etc., see, for instance \cite{GMGT76}, \cite{LT76}, \cite[Ch.~III]{Si71}, 
\cite{Si76} (all in the book \cite{LSW76}), \cite[Sect.~XIII.3]{RS78}, to mention just a few sources. 

We also refer to a variety of other literature on eigenvalue estimates and spectral problems that in one way or another are based on the Birman--Schwinger principle \cite{BS91,GH87,Ho70,Kl82,KS80,KK66,Ne83,Ra80,Se74,Si85,Si77a,RS78}, but cannot make 
any attempt to be complete in this context as the existing literature is of an overwhelming nature. While most of these sources focus on self-adjoint situations, eventually, the Birman--Schwinger technique was extended to non-self-adjoint situations in the context of complex resonances in \cite{AAD01,Si89}, and systematically in \cite{GLMZ05} (and later again in \cite{BGHN16}, \cite{BGHN17},  \cite{GHN15}), adapting factorization techniques developed in Kato \cite{Ka66}, Konno and Kuroda \cite{KK66}, and Howland \cite{Ho70}, primarily in the self-adjoint context. The past 15 years saw enormous interest in various aspects of spectral theory associated with non-self-adjoint problems and we refer, for instance, to \cite{BMNW08,Da07,DH15,DHK09,DHK13,FKV18,Fr11,Fr18,FLLS06,FLS16,FLS11,FS17,Ha11,LS09,Sa10,Sa10a}, again, just a tiny selection of the existing literature without hope of any kind of completeness, and to 
\cite[Sect.~III.9]{BC19}, \cite{LS10} for applications to spectral stability of nonlinear systems. The absence of any eigenvalues (and continuity of the spectrum for $n=3$) under a general smallness condition on the complex-valued potential $V$ in Schr\"odinger operators $H=-\Delta +V$ defined as form sums in $L^2(\bbR^n; d^nx)$, $n \geq 3$, has been proven in \cite{FKV18}. In particular, their treatment for the physically relevant situation $n=3$ permits potentials $V$ with local singularities exceeding those in the Rollnik class; the authors in \cite{FKV18} also discuss magnetic vector potentials. 

The principal purpose of this paper is to prove a generalized Birman--Schwinger principle in the non-self-adjoint context that not only 
focuses on a detailed discussion of geometric and algebraic multiplicities of eigenvalues of the operator of interest 
(e.g., a Schr\"odinger operator) and the associated Birman--Schwinger operator, but also a detailed discussion of the associated eigenvectors 
and the corresponding Jordan chains of generalized eigenvectors. 
In Section 2 we first recall the notions of algebraic and geometric eigenvalue multiplicities and the corresponding Jordan chains of generalized eigenvectors 
of a single Hilbert space operator.  
As a warm up we treat an exactly solvable example of a non-self-adjoint, one-dimensional, 
periodic Schr\"odinger operator, exhibiting algebraic multiplicities of eigenvalues strictly larger than geometric ones in Proposition~\ref{p2.2}.
Afterwards, in Section~\ref{s3} we
then turn to one of our principal topics, algebraic and geometric multiplicities and corresponding Jordan chains of 
generalized eigenvectors associated with zeros of strongly analytic operator-valued functions. This leads directly to 
Section~\ref{s4}, where our main result is formulated. More precisely, in Theorem~\ref{t4.5} we show that if $H_0$ is a closed operator in some Hilbert space $\cH$ and
$V=V_2^*V_1$ is an additive perturbation of $H_0$, then the vectors 
$\{f_0,\dots,f_{k-1}\}$ form a Jordan chain for the perturbed operator $H=H_0+V_2^*V_1$ if and only if the vectors $\{\varphi_0,\dots,\varphi_{k-1}\}$
form a Jordan chain for the operator-valued function (an abstract Birman--Schwinger-type operator family)
\begin{equation}
 \rho(H_0)\ni z\mapsto I_\cK+ V_1(H_0-zI_\cH)^{-1}V_2^*;
\end{equation}
here $\cK$ is some auxiliary Hilbert space and $V_1,V_2$ are (possible unbounded) operators mapping from $\cH$ to $\cK$ that satisfy some additional mild technical 
conditions (cf.\ Hypothesis~\ref{h4.1}). In the context of non-self-adjoint second-order 
elliptic partial differential operators and Dirichlet-to-Neumann maps a similar correspondence was 
found recently in \cite{BE19}; see also \cite[Sect.~7.4.4]{DM17} for a related result in the abstract setting of extension theory of symmetric operators.
The notion of the index of meromorphic operator-valued functions is briefly recalled in Section~\ref{s5} and an application to algebraic multiplicities of meromorphic operator-valued functions is provided in Theorem~\ref{t5.7} and Theorem~\ref{t5.9}. 
Our final Section~\ref{s6} centers around the notion of the essential 
spectrum of closed operators in a Hilbert space, and as a highlight derives a global version of the Weinstein--Aronszajn formula 
(relating algebraic multiplicities of a pair of operators) 
in the non-self-adjoint context in Theorem~\ref{t6.6} and Theorem~\ref{t6.8}.

Finally, we summarize the basic notation used in this paper: $\cH$ and $\cK$ denote separable complex 
Hilbert spaces with scalar products $(\,\cdot\,,\,\cdot\,)_{\cH}$ and $(\,\cdot\,,\,\cdot\,)_{\cK}$, linear in the 
first entry, respectively. 

The Banach spaces of bounded and compact linear operators in $\cH$ are denoted by $\cB(\cH)$ and
$\cB_\infty(\cH)$, respectively. Similarly, the Schatten--von Neumann
(trace) ideals will subsequently be denoted by $\cB_p(\cH)$,
$p \in [1,\infty)$, and the subspace of all finite-rank operators in $\cB_1(\cH)$ will be 
abbreviated by $\cF(\cH)$. Analogous notation $\cB(\cH,\cK)$,
$\cB_\infty (\cH,\cK)$, etc., will be used for bounded, compact, etc.,
operators between two Hilbert spaces $\cH$ and $\cK$. In addition,
$\tr_{\cH}(T)$ denotes the trace of a trace class operator $T\in\cB_1(\cH)$ and
$\det_{p,\cH}(I_{\cH}+S)$ represents the (modified) Fredholm determinant
associated with an operator $S\in\cB_p(\cH)$, $p\in\bbN$ (for $p=1$ we
omit the subscript $1$). Moreover, $\Phi(\cH)$ denotes the set of bounded Fredholm operators on 
$\cH$, that is, the set of operators $T \in \cB(\cH)$ such that $\dim(\ker(T)) < \infty$, $\ran(T)$ is 
closed in $\cH$, and $\dim(\ker(T^*)) < \infty$. The corresponding (Fredholm) index 
of $T \in \Phi(\cH)$ is then given by $\ind(T) = \dim(\ker(T)) - \dim(\ker(T^*))$. 
For a linear operator $T$ we denote by $\dom(T)$, $\ran(T)$ and $\ker(T)$ the domain, range, and kernel (i.e., 
nullspace), respectively. If $T$ is closable, the closure is denoted by $\overline T$. 
The spectrum, point spectrum,  and resolvent set of a closed operator $T$  will be 
denoted by $\sigma(T)$, $\sigma_p(T)$, and $\rho(T)$, respectively. 

The identity matrix in $\bbC^N$, $N \in \bbN$, is written as $I_N$, the corresponding nullmatrix will be 
abbreviated by $0_N$; by $D(z_0; r_0) \subset \bbC$ we denote the open disk with center $z_0$ and radius 
$r_0 > 0$, and by $\partial D(z_0; r_0)$ the corresponding circle; we also employ the notation 
$\bbN_0 = \bbN \cup \{0\}$.

%%%%%%%
%%%%%%%
\section{Eigenvalues and Jordan Chains for (Unbounded) Operators: A Warm Up and a  Non-Self-Adjoint Schr\"odinger Operator} \lb{s2}
%%%%%%%
%%%%%%%

To motivate the objects of interest in this paper a bit, we start by considering a (possibly unbounded) 
operator $A$ in a separable, complex Hilbert space 
$\cK$. Recall first that the vectors $\{\varphi_0,\ldots,\varphi_{k-1}\} \subset  \dom(A)$ form a 
{\it Jordan chain of length} $k$ for $A$ at the eigenvalue $\lambda_0\in\bbC$ if and only if $\varphi_0\not=0$ and
\begin{equation} \lb{2.1a}
 (A- \lambda_0 I_{\cK})\varphi_0=0 \, \text{ and } \, (A - \lambda_0 I_{\cK})\varphi_j = \varphi_{j-1},\quad j\in\{1,\ldots,k-1\}.
\end{equation}
In this case the vector $\varphi_0$ is an eigenvector corresponding to the eigenvalue 
$\lambda_0 \in \sigma_p(A)$ and the vectors $\varphi_1,\ldots,\varphi_{k-1}$
are usually called generalized eigenvectors corresponding to $\lambda_0$. It is clear from \eqref{2.1a} that all generalized eigenvectors are nontrivial.
If $\{\varphi_{0,n}\}_{1\leq n \leq N}$, $N \in \bbN \cup \{\infty\}$, is a basis in $\ker(A - \lambda_0 I_{\cK})$
and $\{\varphi_{0,n},\ldots,\varphi_{k_n-1,n}\} \subset  \dom(A)$, $1\leq n\leq N$, are the corresponding Jordan chains of maximal lengths $k_n \in \bbN$, then the 
{\it algebraic multiplicity} $m_a(\lambda_0; A )$ of the eigenvalue $\lambda_0$ is defined as
\begin{equation}
m_a(\lambda_0; A) = \sum_{n=1}^N k_n.   \lb{2.2a} 
\end{equation}
If one of the eigenvectors has a Jordan chain of arbitrarily long length, then we define 
$m_a(\lambda_0; A) = \infty$.

Next, the {\it geometric multiplicity} of an eigenvalue $\lambda_0 \in \sigma_p(A)$ of $A$, denoted by 
$m_g(\lambda_0; A)$, is given by 
\begin{equation}
m_g(\lambda_0; A) = \dim(\ker(A - \lambda_0 I_{\cK})).
\end{equation}  

In the following paragraph we assume that $A$ is a closed operator in $\cK$ (and hence $\ker(A - z I_{\cK})$ is closed in $\cK$, $z \in \bbC$): Suppose $\lambda_0\in\bbC$ is an isolated point in $\sigma(A)$ and introduce the {\it Riesz projection} $P(\lambda_0;A)$ of $A$ corresponding to $\lambda_0$ by
\begin{equation}
P(\lambda_0;A) = \f{-1}{2\pi i} \ointctrclockwise_{\partial D(\lambda_0; \varepsilon) }
d\zeta \, (A-\zeta I_{\cK})^{-1}, \lb{2.4}
\end{equation}
where $\partial D(\lambda_0; \varepsilon) $ is a counterclockwise oriented circle centered at $\lambda_0$ with sufficiently small radius $\varepsilon>0$ (excluding the rest of $\sigma(A)$). If the Riesz projection
is a finite-rank operator in $\cK$, then \\[1mm] 
$(i)$ $\lambda_0$ is an eigenvalue of $A$, \\[1mm]
and \\[1mm] 
$(ii)$ $\ran(P(\lambda_0;A))$ coincides with the algebraic eigenspace of $A$ at $\lambda_0$. In 
this case one obtains for the algebraic multiplicity $m_a(\lambda_0; A)$ of the eigenvalue $\lambda_0$ of $A$ 
\begin{equation} 
m_a(\lambda_0; A) = \dim(\ran(P(\lambda_0;A))) = {\tr}_{\cK}(P(\lambda_0;A))    \lb{2.5} 
\end{equation} 
(see, e.g., \cite[Sect.~XV.2]{GGK90}), \cite[Sect.~I.2]{GK69}, \cite[Sect.~III.6.5]{Ka80}), and 
\begin{equation} 
m_g(\lambda_0; A) \leq m_a(\lambda_0; A).    \lb{2.6a} 
\end{equation} 

Following a standard practice (particularly, in the special context of self-adjoint operators $A$ in $\cK$), we now introduce the {\it discrete spectrum} of $A$ in $\cK$ by
\begin{align} 
\begin{split} 
& \sigma_d(A) = \{\lambda \in \sigma_p(A) \,|\, \text{$\lambda$ is an isolated point of $\sigma(A)$}   \\
& \hspace*{3.3cm} \text{with $\dim(\ran(P(\lambda_0;A))) < \infty$}\}.     \lb{2.7} 
\end{split}
\end{align}
Any element of $\sigma_d(A)$ in \eqref{2.7} is called a {\it discrete eigenvalue of $A$.}

%%%%%%%
\begin{remark} \lb{r2.1}
Assume that $A$ is closed in $\cK$. Then what we called a {\it discrete eigenvalue of $A$} in \eqref{2.7} 
(see, e.g., \cite[p.~13]{RS78}), is also called an 
{\it eigenvalue of finite-type of $A$} (cf.\, e.g.,  \cite[p.~326]{GGK90}), or, a 
{\it normal eigenvalue of $A$} (see, e.g., \cite[p.~9]{GK69}). 
${}$ \hfill $\diamond$
\end{remark}
%%%%%%%
  
Next, to illustrate the notion of Jordan chains with a concrete and exactly solvable example of a non-self-adjoint, one-dimensional, 
periodic Schr\"odinger operator, exhibiting algebraic multiplicities of eigenvalues strictly larger than geometric ones, we now 
develop the case of the exactly solvable exponential potential in some detail:
  
%%%%%%
\begin{proposition} \lb{p2.2}
Let $\alpha \in \bbC$, consider the potential
\begin{equation}
V(\alpha,x) = \alpha^2 e^{ix}, \quad x \in \bbR,    \lb{2.8A}
\end{equation}
and introduce the associated Schr\"odinger differential expression\footnote{We chose $\alpha^2$ instead of $\alpha$ to be the coupling constant in \eqref{2.9A} to avoid taking numerous square roots of the coupling constant later on (cf.~\eqref{2.12A}).} 
\begin{equation}
\tau (\alpha) = - \f{d^2}{dx^2} + \alpha^2 e^{ix}, \quad x \in \bbR,     \lb{2.9A}
\end{equation}
and the underlying periodic Schr\"odinger operator $H_p(\alpha)$ in $L^2([0,2\pi]; dx)$,  
\begin{align}
& (H_p(\alpha) f)(x) = (\tau (\alpha) f)(x) \text{ for a.e. $x \in [0, 2 \pi]$,} \no \\
& \, f \in \dom(H_p(\alpha)) = \big\{g \in L^2([0,2\pi]; dx) \, \big| \, g, g' \in AC([0, 2\pi]); \\ 
& \hspace*{3.26cm} g(0) = g(2 \pi), \, g'(0) = g'(2 \pi); 
\, g'' \in L^2([0,2\pi]; dx)\big\}.    \no 
\end{align}
Then $\sigma(H_p(\alpha))$ is purely discrete,
\begin{equation}
\sigma(H_p(\alpha)) = \{z \in \bbC \, | \, D(z) = 1\} = \big\{m^2\big\}_{m \in \bbN_0},     \lb{2.11A} 
\end{equation} 
with corresponding $($non-normalized\,$)$ eigenfunctions $y\big(m^2, \dott\big) \in \dom(H_p(\alpha))$ explicitly given by 
\begin{equation}
y\big(m^2,x\big) =  J_{2m} \big(2 \alpha e^{ix/2}\big),
\quad x \in [0, 2\pi], \; m \in \bbN_0     \lb{2.12A} 
\end{equation}
$($with $J_{\nu}(\dott)$ the regular Bessel function of order $\nu \in \bbC$, see 
\cite[Sect.~9.1]{AS72}$)$, if $\alpha \neq 0$. 
In fact, for $m \in \bbN_0$, the associated kernel $($i.e., geometric eigenspace$)$ of $H_p(\alpha) - m^2 I_{ L^2([0,2\pi]; dx)}$ is one-dimensional, 
\begin{equation}
\ker \big(H_p(\alpha) - m^2 I_{ L^2([0,2\pi]; dx)}\big) 
= \big\{c_m y\big(m^2, \dott\big) \, \big| \, c_m \in \bbC\big\},   
\end{equation}
if $\alpha \neq 0$.
If $\alpha = 0$, then 
\[
\ker \big(H_p(\alpha) - m^2 I_{ L^2([0,2\pi]; dx)}\big) 
= \big\{ x \mapsto c_m e^{i m x} + d_m e^{- i m x} \, \big| \,  c_m, d_m \in \bbC \big\}, \quad m \in \bbN_0, 
\]
is two dimensional, except if $m = 0$, when it is one-dimensional. 
Next, introducing 
\begin{align}
\begin{split} 
\dot y\big(m^2,x\big) = [(2m-1)!] \sum_{k=0}^{2m-1}  [(2m-k) (k!)]^{-1} 
\big[\alpha e^{ix/2}\big]^{k-2m} J_k\big(2 \alpha e^{ix/2}\big),& \\
x \in [0,2 \pi], \;  m \in \bbN,&    \lb{2.13A} 
\end{split}
\end{align}
then $\dot y\big(m^2,\dott\big) \in \dom\Big( \big(H_p(\alpha) - m^2 I_{ L^2([0,2\pi]; dx)}\big)^2\Big)$, 
$m \in \bbN$, and 
\begin{align}
& \big(H_p(\alpha) - m^2 I_{ L^2([0,2\pi]; dx)}\big) \dot y\big(m^2, \dott \big) = y\big(m^2, \dott\big), 
\quad m \in \bbN, \\
& \big(H_p(\alpha) - m^2 I_{ L^2([0,2\pi]; dx)}\big)^2 \dot y\big(m^2, \dott \big) = 0, \quad m \in \bbN.  
\end{align}
Moreover, for each $m \in \bbN$, the algebraic eigenspace of $H_p(\alpha)$ corresponding to the eigenvalue 
$m^2$ is two-dimensional and given by
\begin{equation}
\big\{c_m y\big(m^2,\dott\big) + d_m \dot y\big(m^2,\dott\big) \, \big| \, c_m, d_m \in \bbC \big\},   \lb{2.16A} 
\end{equation}
exhibiting the Jordan chain $\big\{ y\big(m^2, \dott \big),  \dot y\big(m^2, \dott \big) \big\}$ $($cf. \eqref{2.1a}$)$ corresponding to the eigenvalue $m^2$, $m \in \bbN$. In particular, one obtains
\begin{align}
\begin{split} 
& m_g (0; H_p(\alpha)) = 1, \\
& m_g \big(m^2; H_p(\alpha)\big) = 
\begin{cases} 1, & m \in \bbN, \; \alpha \in \bbC \backslash \{0\},  \\  
2, & m \in \bbN, \; \alpha = 0, \end{cases}     \lb{2.17A}
\end{split} \\ 
\begin{split} 
& m_a(0; H_p(\alpha)) = 1, \\   
& m_a \big(m^2; H_p(\alpha)\big) = 2, \quad m \in \bbN, \; \alpha \in \bbC.    \lb{2.18A}
\end{split} 
\end{align}

In \eqref{2.11A} we abbreviated the underlying Floquet 
discriminant by $D(\, \cdot \,)$ and note that actually,
\begin{equation}
D(z) = \cos\big(z^{1/2} 2\pi\big), \quad z \in \bbC,      \lb{2.22A} 
\end{equation} 
in this particular example.
\end{proposition}
%%%%%%%
\begin{proof}
The (self-adjoint) case $\alpha = 0$ can be done by elementary means, since
the differential equation  $- y''(z,x) = z y(z,x)$, is explicitly solvable in terms of the exponential
functions $y_{\pm}(z,x) = e^{\pm i z^{1/2} x}$, $x \in \bbR$.
So in the sequel we may (and do) assume that $\alpha \neq 0$.

We start by observing that the general solution of the differential equation 
$\tau(\alpha) y(z, \dott) = z y(z, \dott)$, $z \in \bbC$, is of the explicit form (see, e.g., \cite[No.~9.1.54]{AS72})
\begin{equation}
y(z,x) = \begin{cases} c_1(z) J_{2z^{1/2}} \big(2 \alpha e^{i x/2}\big) 
+ c_2(z) J_{- 2z^{1/2}} \big(2 \alpha e^{ix/2}\big), 
& 2 z^{1/2} \in \bbC \backslash \bbN_0, \\[1mm] 
d_1(n) J_{n} \big(2 \alpha e^{i x/2}\big) + d_2(n) Y_{n} \big(2 \alpha e^{i x/2}\big), & 2 z^{1/2} =n, \, n \in \bbN_0, 
\end{cases}   \lb{2.7aa} 
\end{equation}
with $J_{\nu}(\dott), Y_{\nu}(\dott)$, $\nu \in \bbC$, the standard regular and 
irregular Bessel functions (see again, e.g., \cite[Sect.~9.1]{AS72}), and $c_j(z), d_j(n) \in \bbC$, $j = 1,2$, 
appropriate constants with respect to $x \in \bbR$. 

Next, introducing the special fundamental system of solutions $\phi_0(z, \dott), \theta_0(z, \dott)$ of 
$\tau (\alpha) y (z, \dott) = z y(z, \dott)$, entire in the parameter $z \in \bbC$ and uniquely characterized 
by the initial conditions
\begin{equation}
\phi_0 (z,0) = 0, \; \phi_0 ' (z,0) = 1,  \quad \theta_0 (z,0) = 1, \; \theta_0 '(z,0) = 0, \quad z \in \bbC,
\end{equation}
one obtains (for $2 z^{1/2} \in \bbC \backslash \bbN_0$, $x \in \bbR$)
\begin{align}
& \phi_0(z,x) = \f{i \pi}{\sin\big(z^{1/2} 2 \pi\big)} \big[J_{2z^{1/2}}(2\alpha) J_{- 2z^{1/2}} \big(2 \alpha e^{ix/2}\big) 
- J_{-2z^{1/2}}(2\alpha) J_{2z^{1/2}} \big(2 \alpha e^{ix/2}\big)\big],   \lb{2.9a} \\
& \theta_0(z,x) = \f{\pi \alpha}{\sin\big(z^{1/2} 2 \pi\big)} \big[J_{2z^{1/2}}' (2\alpha) J_{- 2z^{1/2}} \big(2 \alpha e^{ix/2}\big) 
- J_{-2z^{1/2}}' (2\alpha) J_{2z^{1/2}} \big(2 \alpha e^{ix/2}\big)\big],     \lb{2.10a} 
\end{align}
and (for $2z^{1/2} = n \in \bbN_0$, $x \in \bbR$)
\begin{equation}
\begin{split}
& \phi_0(z,x) =  i \pi \big[ Y_n(2\alpha) J_n \big(2 \alpha e^{ix/2}\big)
                             - J_n(2\alpha) Y_n \big(2 \alpha e^{ix/2}\big) \big],   \\
& \theta_0(z,x) = \pi \alpha \big[ Y_n'(2\alpha) J_n \big(2 \alpha e^{ix/2}\big)
                                   - J_n'(2\alpha) Y_n \big(2 \alpha e^{ix/2}\big) \big] . 
\end{split}
                                   \end{equation}
In particular,
\begin{equation}
W(\theta_0(z,\dott), \phi_0(z,\dott)) = 1, \quad z \in \bbC, 
\end{equation}
and the monodromy matrix $\cM(\dott)$ associated with $\tau(\alpha)$ is thus of the type
\begin{equation}
\cM(z) = \begin{pmatrix} \theta_0(z,2 \pi) & \phi_0(z,2 \pi) \\ \theta_0'(z,2 \pi) & \phi_0'(z,2 \pi) \end{pmatrix}, 
\quad z \in \bbC,
\end{equation}
with Floquet discriminant (i.e., Lyapunov function) $D(\dott)$ given by
\begin{equation}
D(z) = \tr_{\bbC^2} (\cM(z)) / 2 = [\theta_0(z,2 \pi) + \phi_0'(z,2 \pi)]/2, \quad z \in \bbC. 
\end{equation} 
Employing standard properties of the Bessel functions in \eqref{2.9a}, \eqref{2.10a} and \cite[No.~9.1.35]{AS72} one confirms that actually, 
$D(z) = \cos\big(z^{1/2} 2\pi\big)$, $z \in \bbC$, that is, \eqref{2.22A} holds (see also \cite{De??}, 
\cite{Ga80}, \cite{GU83}, \cite{GU83a}, \cite{PT88}, \cite{PT91}, \cite{Sh03}, \cite{Sh04}).

To determine the spectrum of $H_p(\alpha)$ one recalls that by general Floquet theory 
\begin{equation}
\sigma(H_p(\alpha)) = \{z \in \bbC \, | \, D(z) = 1\}. 
\end{equation}
Furthermore, one verifies that the resolvent of $H_p(\alpha)$ is a Hilbert--Schmidt operator and hence
$\sigma(H_p(\alpha))$ is purely discrete, and given by
\begin{equation}
\sigma(H_p(\alpha)) = \big\{m^2\big\}_{m \in \bbN_0},  
\end{equation} 
with corresponding eigenfunctions $y_m \in \dom(H_p(\alpha))$ explicitly given by \eqref{2.12A}. 
In fact, using \eqref{2.7aa} it follows that the associated kernel of $H_p(\alpha) - m^2 I_{ L^2([0,2\pi]; dx)}$ is one-dimensional, 
since (cf.\ \cite[No.~9.1.36]{AS72})
\begin{align}
\begin{split} 
Y_{\nu}\big(\zeta e^{i\pi}\big) &= e^{- i \nu \pi} Y_{\nu}(\zeta) 
+ 2 i \cos(\nu \pi) J_{\nu}(\zeta), \quad \zeta \in \bbC \backslash (-\infty,0],     \\
Y_{\nu}'\big(\zeta e^{i\pi}\big) &= - e^{- i \nu \pi} Y_{\nu}'(\zeta) 
- 2 i \cos(\nu \pi) J_{\nu}'(\zeta), \quad \zeta \in \bbC \backslash (-\infty,0],     
\end{split}
\end{align}
(and similarly on the cut $\zeta \in (-\infty,0]$) and hence 
\begin{align}
\begin{split}
Y_{2m}(2 \alpha e^{i\pi}) &= Y_{2m}(2 \alpha) + 2i J_{2m}(2 \alpha),    \\
Y_{2m}'(2 \alpha e^{i\pi}) &= - Y_{2m}'(2 \alpha) - 2i J_{2m}'(2 \alpha), 
\end{split}
\end{align}
that is, $Y_{2m}(2 \alpha e^{i x/2})$, $m \in \bbN_0$, cannot satisfy the periodic boundary conditions at 
$x =0$ and $2\pi$ for elements in $\dom(H_p(\alpha))$ (see also, \cite[Sect.~9.5]{AS72}). On the other hand, 
employing 
\begin{equation}
J_{2m} \big(\zeta e^{i \pi}\big) = J_{2m}(\zeta), \quad 
J_{2m}' \big(\zeta e^{i \pi}\big) = - J_{2m}' (\zeta),  \quad 
\zeta \in \bbC    \lb{2.33A} 
\end{equation}
(cf.~\cite[9.1.35]{AS72}), $J_{2m}(2 \alpha e^{i x/2})$, $m \in \bbN_0$, clearly satisfies these periodic boundary conditions at $x =0$ and $2\pi$. 

To determine the algebraic multiplicity of the periodic eigenvalues $m^2$, $m \in \bbN_0$, we recall the fact (see, e.g., \cite{GW95}),
\begin{align}
& {\det}_{L^2([0,2\pi]; dx)} \big((H_p(\alpha) - z I_{L^2([0,2\pi]; dx)}) 
(H_p(\alpha) - z_0 I_{L^2([0,2\pi]; dx)})^{-1}\big)    \no \\
& \quad = {\det}_{L^2([0,2\pi]; dx)} \big(I_{L^2([0,2\pi]; dx)} - (z - z_0) 
(H_p(\alpha) - z_0 I_{L^2([0,2\pi]; dx)})^{-1}\big)   \no \\ 
& \quad = \f{D(z) - 1}{D(z_0) - 1} 
= \f{\cos\big(z^{1/2} 2 \pi\big) - 1}{\cos\big(z_0^{1/2} 2 \pi\big) - 1}, \quad z \in \bbC, \; 
z_0 \in \bbC \big \backslash \big\{m^2\big\}_{m \in \bbN_0},     \lb{2.32}
\end{align}
where $\det_{\cH} (I_{\cH} + T)$ represents the Fredholm determinant in connection with the trace class operator $T \in \cB_1(\cH)$ in the complex, 
separable Hilbert space $\cH$. Thus, the algebraic multiplicity of the eigenvalue $m^2$, $m \in \bbN_0$, of $H_p(\alpha)$ coincides with the 
order of the zero of 
\begin{equation}
{\det}_{L^2([0,2\pi]; dx)} \big(I_{L^2([0,2\pi]; dx)} - (z - z_0) 
(H_p(\alpha) - z_0 I_{L^2([0,2\pi]; dx)})^{-1}\big)
\end{equation}
at the point $z = m^2$, and hence also with the order of the zero of $D(z) - 1 = \cos\big(z^{1/2} 2 \pi\big) - 1$ at $z = m^2$. This proves \eqref{2.17A}, \eqref{2.18A}, 
see also \cite{De??}, \cite{Ga80}, \cite{GU83}, \cite{GU83a}, \cite{PT88}, \cite{PT91}.

It remains to determine the Jordan chains associated with $m^2$, $m \in \bbN$. For this purpose we fix 
$m \in \bbN$ and note that
\begin{equation}\label{asas}
- y''(z,x) + \big[\alpha^2 e^{ix} -z\big] y(z,x) = 0, 
\end{equation}
implies (with $\dot {}$ abbreviating $d/dz$), 
\begin{equation}
- \dot y''(z,x) + \big[\alpha^2 e^{ix} -z\big] \dot y(z,x) = y(z,x). 
\end{equation}
We note that for $z\in\bbC$ the function
\begin{equation}
y(z,x) = 
c_1(z) J_{2z^{1/2}} \big(2 \alpha e^{i x/2}\big) + c_2(z) Y_{2z^{1/2}} \big(2 \alpha e^{i x/2}\big) 
\end{equation}
is a solution of \eqref{asas} (see, e.g., \cite[No.~9.1.54]{AS72} or \eqref{2.7aa} for $z=m^2$) and,
from now on for simplicity, 
we agree to choose $c_1(z)=1$, $z \in \bbC$, and that $c_2(\dott)$ is differentiable (without loss of generality) and that it satisfies 
\begin{equation}
c_2\big(m^2\big) = 0,    \lb{2.33}
\end{equation}
in accordance with the boundary conditions in $\dom(H_p(\alpha))$; $c_2(\dott)$ will explicitly be chosen  
in \eqref{2.41A}. One then computes 
\begin{align}
\begin{split} 
\dot y(z,x) &= 
 \dot c_2(z) Y_{2 z^{1/2}}\big(2 \alpha e^{ix/2}\big)    \\
& \quad + (\partial/\partial z) J_{2 z^{1/2}}\big(2 \alpha e^{ix/2}\big) 
+ c_2(z) (\partial/\partial z) Y_{2 z^{1/2}}\big(2 \alpha e^{ix/2}\big),
\end{split} 
\end{align}
and hence 
\begin{align}
\dot y\big(m^2,x\big) &= \dot c_2 \big(m^2\big) Y_{2m}\big(2 \alpha e^{ix/2}\big) + (\partial/\partial z) J_{2 z^{1/2}}\big(2 \alpha e^{ix/2}\big)\big|_{z=m^2}   \no\\
&= \dot c_2 \big(m^2\big) Y_{2m}\big(2 \alpha e^{ix/2}\big) 
+ \Bigg[[\pi/(2m)] Y_{2m}\big(2 \alpha e^{ix/2}\big)    \no \\ 
& \quad + [(2m-1)!] \sum_{k=0}^{2m-1} 
\f{\big[\alpha e^{ix/2}\big]^{k-2m}}{(2m-k) (k!)} J_k \big(2 \alpha e^{ix/2}\big)\Bigg],   \lb{2.35} 
\end{align}
where we employed \eqref{2.33} and (cf.~\cite[No.~9.1.66]{AS72}) 
\begin{align}
\begin{split} 
[(\partial/\partial \nu) J_{\nu}(\zeta)\big|_{\nu = 2m} &= (\pi/2) Y_{2m}(\zeta)   \\
& \quad + 2^{-1} [(2m)!] \sum_{k=0}^{2m-1} [(2m-k) (k!)]^{-1} (\zeta/2)^{k-2m} J_k(\zeta).
\end{split} 
\end{align}
Next, we choose,  
\begin{equation}
\dot c_2\big(m^2\big) = - \pi/(2m) \, \text{ and } \, c_2(z) = - \pi \big(z - m^2\big)/(2m), \quad m \in \bbN, 
\lb{2.41A} 
\end{equation}
to eliminate $Y_{2m}(\dott)$ in \eqref{2.35}, finally resulting in
\begin{equation}
\dot y\big(m^2,x\big) = [(2m-1)!] \sum_{k=0}^{2m-1}  [(2m-k) (k!)]^{-1} 
\big[\alpha e^{ix/2}\big]^{k-2m} J_k\big(2 \alpha e^{ix/2}\big), \quad m \in \bbN.   \lb{2.42A} 
\end{equation}
Exploiting a slight extension of \eqref{2.33A}, 
\begin{equation}
J_k \big(\zeta e^{i \pi}\big) = (-1)^k J_k (\zeta), \quad 
J_k' \big(\zeta e^{i \pi}\big) = (-1)^{k+1} J_k' (\zeta),  \quad 
\zeta \in \bbC, \; k \in \bbN_0    
\end{equation}
(cf.~\cite[9.1.35]{AS72}), one verifies that 
\begin{equation}
\dot y\big(m^2,0\big) = \dot y\big(m^2,2 \pi\big), \quad 
\dot y'\big(m^2,0\big) = \dot y'\big(m^2,2 \pi\big),  
\end{equation}
that is, $\dot y\big(m^2,\dott\big)$, $m \in \bbN$, in \eqref{2.42A} satisfy the boundary conditions in $\dom(H_p(\alpha))$. 
Moreover, employing the identity (cf.~\cite[No.~9.1.27]{AS72}), 
\begin{align}
& J_{\nu}'(\zeta) = - J_{\nu + 1}(\zeta) + (\nu/\zeta) J_{\nu}(\zeta), \quad\nu\in\bbC, 
\end{align}
and using the fact that 
$J_{\nu}(\dott)$ satisfies the second-order differential equation 
\begin{equation}
 \frac{d^2}{dx^2} J_{\nu}\big(2 \alpha e^{ix/2}\big)=\big[\alpha^2 e^{ix} - \big(\nu^2 \big/ 4\big)\big] J_{\nu}\big(2 \alpha e^{ix/2}\big), \quad \nu\in\bbC,\,\,x \in \bbR, 
\end{equation}
(this follows from \cite[No.~9.1.54]{AS72}),
one verifies that  
\begin{equation}
- \dot y''\big(m^2,x\big) + \big[\alpha^2 e^{ix} - m^2\big] \dot y\big(m^2,x\big) = y\big(m^2,x\big), 
\quad m \in \bbN, 
\end{equation}
as follows:
\begin{align}
& \bigg[- \f{d^2}{dx^2} + \alpha^2 e^{ix}\bigg] \dot y\big(m^2,x\big) = \alpha^2 e^{ix} \dot y\big(m^2,x\big) 
\no \\
& \qquad - [(2m-1)!] \sum_{k=0}^{2m-1} \f{\alpha^{k-2m}}{(2m-k) [k!]} \bigg[- \f{(k-2m)^2}{4} e^{i(k-2m)x/2} 
J_k\big(2 \alpha e^{ix/2}\big)     \no \\
& \hspace*{5.8cm} - \alpha (k-2m) e^{i(k+1-2m)x/2} J_k'\big(2 \alpha e^{ix/2}\big)     \no \\ 
& \hspace*{5.8cm} + e^{i(k-2m)x/2} \big[\alpha^2 e^{ix} - \big(k^2 \big/ 4\big)\big] J_k\big(2 \alpha e^{ix/2}\big)\bigg]   \no \\
& \quad = m^2 \dot y\big(m^2,x\big) 
+ [(2m-1)!] \sum_{k=0}^{2m-1} \f{\alpha^{k-2m} (k-2m)}{(2m-k) [k!]} e^{i(k-2m)x/2}    \no \\
& \hspace*{5.5cm} \times 
\Big[(k/2) J_k\big(2 \alpha e^{ix/2}\big) + \alpha e^{i x/2} J_k'\big(2 \alpha e^{i x/2}\big)\Big]     \no \\
& \quad = m^2 \dot y\big(m^2,x\big) - [(2m-1)!] \Bigg[\sum_{k=1}^{2m-1} 
\f{\big[\alpha e^{ix/2}\big]^{k-2m}}{[(k-1)!]} J_k\big(2 \alpha e^{ix/2}\big)   \no \\ 
& \hspace*{4.9cm} - \sum_{k=0}^{2m-1} 
\f{\big[\alpha e^{ix/2}\big]^{k+1-2m}}{[k!]} J_{k+1}\big(2 \alpha e^{ix/2}\big)  \Bigg]    \no \\
& \quad = m^2 \dot y\big(m^2,x\big) + J_{2m}\big(2 \alpha e^{ix/2}\big)    \no \\
& \quad = m^2 \dot y\big(m^2,x\big) + y\big(m^2,x\big),
\end{align}
proving \eqref{2.13A}--\eqref{2.18A}. 
\end{proof}
%%%%%%

%%%%%%
\begin{remark} \lb{r2.2}
The fact that $m_a\big(m^2; H_p(\alpha)\big)=2$, $m\in\bbN$, can also independently be established as follows. Suppose that $\Omega \subseteq \bbC$ open, 
$\cH$ is a complex separable Hilbert space, 
$T: \Omega \to \cB_1(\cH)$ is analytic with respect to the trace norm $\|\dott\|_{\cB_1(\cH)}$, and $A$ is a densely defined, closed operator 
in $\cH$ such that $(A - z_0 I_{\cH})^{-1} \in \cB_1(\cH)$ for some (and hence for all) $z_0 \in \rho(A)$. Then the trace formula (cf.\ \cite[eq.~(IV.1.14), p.~163]{GK69})
\begin{equation}
\tr_{\cH}\big((I_{\cH} - T(z))^{-1} T'(z)\big) = - \f{d}{dz} \ln({\det}_{\cH}(I_{\cH} - T(z))), \quad 
z \in \Omega,
\end{equation}
applied to the special case 
\begin{align}
\begin{split} 
T_A(z) = I_{\cH} - (A - z I_{\cH}) (A - z_0 I_{\cH})^{-1} = (z - z_0) (A - z_0 I_{\cH})^{-1} \in \cB_1(\cH),& \\
z_0 \in \rho(A), \; z \in \bbC,&
\end{split} 
\end{align}
yields
\begin{align}
& - \f{d}{dz} \ln\big({\det}_{\cH}\big((I_{\cH} - (z - z_0) (A - z_0 I_{\cH})^{-1}\big)\big)   \no \\ 
& \quad = - \f{d}{dz} \ln({\det}_{\cH}(I_{\cH} - T_A(z)))    \no \\ 
& \quad = \tr_{\cH} \big((I_{\cH} - T_A(z))^{-1} T_A'(z)\big)   \no \\
& \quad = \tr_{\cH} \Big(\big[(A - z I_{\cH}) (A - z_0 I_{\cH})^{-1}\big]^{-1} (A - z_0 I_{\cH})^{-1}\Big)  \no \\
& \quad = \tr_{\cH}\big(\big[(A - z_0 I_{\cH}) (A - z I_{\cH})^{-1}\big] (A - z_0 I_{\cH})^{-1}\big)   \no \\
& \quad = \tr_{\cH}\big((A - z I_{\cH})^{-1}\big), \quad z_0, z \in \rho(A),      \lb{2.50} 
\end{align}
employing cyclicity of the trace.

An application of \eqref{2.50} to $A = H_p(\alpha)$ together with \eqref{2.32} thus implies for $z_0 \in \rho(H_p(\alpha)$) 
\begin{align}
& {\tr}_{L^2([0,2\pi];dx)} \big((H_p(\alpha) - z I_{L^2([0,2\pi];dx)})^{-1}\big)    \no \\
& \quad = - \f{d}{dz} \ln\big({\det}_{L^2([0,2\pi];dx)} \big(I_{L^2([0,2\pi];dx)} - (z - z_0) (H_p(\alpha) - z_0 I_{L^2([0,2\pi];dx)})^{-1}\big)\big)    \no \\
& \quad = - \f{d}{dz} \ln\left(\frac{D(z) - 1}{D(z_0) -1}\right)   \no \\
& \quad = \f{\dot D(z)}{1 - D(z)}, \quad z \in \bbC \backslash \big\{m^2\big\}_{m \in \bbN_0}. \lb{2.51} 
\end{align}
(One notes that the 2nd and 3rd lines in \eqref{2.51} are independent of $z_0$).

Thus, one confirms (for $0 < \varepsilon$ sufficiently small and $m \in \bbN$)
\begin{align}
m_a\big(m^2;H_p(\alpha)\big) &= {\tr}_{L^2([0,2\pi];dx)}\big(P\big(m^2;H_p(\alpha)\big)\big) \no \\
&= {\tr}_{L^2([0,2\pi];dx)}\bigg(\f{-1}{2 \pi i} \ointctrclockwise_{\partial D(m^2; \varepsilon)}
d\zeta \, (H_p(\alpha)-\zeta I_{L^2([0,2\pi];dx)})^{-1}\bigg)  \no \\
&= \f{1}{2 \pi i} \ointctrclockwise_{\partial D(m^2; \varepsilon)} d\zeta \, 
\f{\dot D(\zeta)}{D(\zeta) - 1}     \no \\ 
&= \f{1}{2 i} \ointctrclockwise_{\partial D(m^2; \varepsilon)} d\zeta \, \zeta^{-1/2} 
\cot\big(\zeta^{1/2}\pi\big)    \no \\ 
&= \f{1}{2 i} \ointctrclockwise_{\partial D(m^2; \varepsilon)} d\zeta \, 
\f{2}{\pi} \bigg[\f{1}{\zeta - m^2} + \Oh(1)\bigg]    \no \\ 
&= 2, \quad m \in \bbN.
\end{align}
Here the term $\Oh(1)$ abbreviates an analytic function in an open neighborhood of 
$D(m^2; \varepsilon)$ and we employed the elementary identity
\begin{equation}
\f{\sin\big(\zeta^{1/2} 2\pi\big)}{1 - \cos\big(\zeta^{1/2} 2\pi\big)} 
= \cot\big(\zeta^{1/2}\pi\big), \quad \zeta \in \bbC \backslash \big\{m^2\big\}_{m \in \bbN_0}.
\end{equation}
${}$ \hfill $\diamond$
\end{remark} 
%%%%%%

%%%%%%
\begin{remark} \lb{r2.2a}
Without going into further details, we note that the antiperiodic Schr\"odinger operator $H_{ap}(\alpha)$ in $L^2([0,2\pi]; dx)$ is defined by 
\begin{align}
& (H_{ap}(\alpha) f)(x) = (\tau(\alpha) f)(x) \text{ for a.e. $x \in [0, 2 \pi]$,} \no \\
& \, f \in \dom(H_{ap}(\alpha)) = \big\{g \in L^2([0,2\pi]; dx) \, \big| \, g, g' \in AC([0, 2\pi]);   \lb{2.19A} \\ 
& \hspace*{3.4cm} g(0) = - g(2 \pi), \, g'(0) = - g'(2 \pi); 
\, g'' \in L^2([0,2\pi]; dx)\big\},      \no 
\end{align}
and one obtains in close analogy to \eqref{2.11A}, \eqref{2.17A}, and \eqref{2.18A},  
\begin{align}
& \sigma(H_{ap}(\alpha)) = \{z \in \bbC \, | \, D(z) = -1\} = \big\{[m - (1/2)]^2\big\}_{m \in \bbN},   \lb{2.20A} \\
& m_g \big([m - (1/2)]^2; H_{ap}(\alpha)\big) = 1, \quad m_a \big([m - (1/2)]^2; H_{ap}(\alpha)\big) = 2,  
\quad m \in \bbN.      \lb{2.21A} 
\end{align}
${}$ \hfill $\diamond$
\end{remark} 
%%%%%%

Finally, we briefly mention a Floquet theoretic result for the corresponding periodic operator acting on 
the real line which is an immediate consequence of Proposition~\ref{p2.2} (and its proof):

%%%%%%
\begin{corollary} \lb{c2.3} 
Given $V(\alpha,\dott)$ and $\tau(\alpha)$, $\alpha \in \bbC$, as in \eqref{2.8A} and \eqref{2.9A},  we introduce  the corresponding periodic Schr\"odinger operator $H(\alpha)$ in $L^2(\bbR; dx)$ via 
\begin{align}
& (H(\alpha) f)(x) = (\tau (\alpha) f)(x) \text{ for a.e. $x \in \bbR$,}     \no \\
& \, f \in \dom(H(\alpha)) = \big\{g \in L^2(\bbR; dx) \, \big| \, g, g' \in AC_{loc}(\bbR); \, g'' \in L^2(\bbR; dx)\big\}    \\
& \hspace*{2.4cm} = H^2(\bbR).    \no 
\end{align}
Then 
\begin{equation}
\sigma(H(\alpha)) = [0,\infty),  \quad \alpha \in \bbC, 
\end{equation}
equivalently, one obtains the remarkable fact that the spectrum of $H(\alpha)$ is independent of 
$\alpha \in \bbC$ and hence equals that of $H(0)$, where 
\begin{equation}
H(0) = - d^2/dx^2, \quad \dom(H(0)) = H^2(\bbR). 
\end{equation}
\end{corollary}
%%%%%%
\begin{proof}
Standard Floquet theory in the non-self-adjoint context $($see, e.g., \cite{GW95} and the literature cited therein$)$ implies  
\begin{equation}
\sigma(H(\alpha)) = \{z \in \bbC \, | \, D(z) \in [-1,1]\} = [0,\infty),  
\end{equation}
see also \cite{Ga80}, \cite{GU83}, \cite{GU83a}, \cite{PT88}, \cite{PT91}, \cite{Sh03}, \cite{Sh04}. 
\end{proof}
%%%%%%

%%%%%%
\begin{remark} \label{rem222}
The fact that the exponential potential is exactly solvable in terms of Bessel functions is of course well-known, see, for instance, \cite{De??}, 
\cite[Problem~75, p. 196]{Fl94}. The explicit representations of generalized eigenvectors \eqref{2.13A} and algebraic eigenspace 
\eqref{2.16A} appear to be new. Generalizations to appropriate superpositions of exponentials of the form $V(x) = \sum_{n \in \bbN} \alpha_n e^{i n x}$, 
$x \in [0, 2 \pi]$ were studied in \cite{Ga80}, \cite{GU83a} (see also \cite{GU83}), \cite{PT88}, \cite{PT91}, \cite{Sh03}. ${}$ \hfill $\diamond$
\end{remark}
%%%%%%

%%%%%%%
%%%%%%%
\section{Algebraic and Geometric Multiplicities and Jordan Chains for the Zeros of Strongly Analytic Operator-Valued  Functions} \lb{s3}
%%%%%%%
%%%%%%%

In this section we study families of operators, or operator-valued functions rather than a fixed operator as in 
Section~\ref{s2}. Let $\cK$ be a separable, complex Hilbert space and assume that $z\mapsto A(z)$ is a function defined on some open
set $\Omega\subset\bbC$ such that for all $z\in\Omega$ the values $A(z)$ are linear operators in $\cK$ and 
$\dom (A(z))=\cD$ for all $z\in\Omega$. 
In addition, assume that for all $\varphi\in\cD$, the $\cK$-valued function $z\mapsto A(z)\varphi$ is analytic in $\Omega$;
the $\ell$-th derivative of $A(\dott)\varphi$ at $z\in\Omega$ is denoted by $A^{(\ell)}(z)\varphi$. 
Under these assumptions we can extend the notion of Jordan chains of $\cB(\cK)$-valued analytic operator functions 
due to M. V. Keldysh from \cite{Ke51} (see also \cite[Sect.~II.11]{Ma88}).

%%%%%%
\begin{definition}\lb{d3.1}
Suppose that $A(\dott)$ is a strongly analytic function defined on some open set $\Omega\subset\bbC$ with 
$\cD=\dom (A(z)) \subseteq \cK$, $z\in\Omega$ (i.e., for all $\varphi\in\cD$, $z\mapsto A(z)\varphi$ is 
analytic in $\Omega$), and let $\lambda_0 \in \Omega$. \\[1mm] 
$(i)$ Fix $k \in \bbN$ and $\varphi_0,\ldots,\varphi_{k-1} \in \cD$.
We say that the vectors $\{\varphi_0,\ldots,\varphi_{k-1}\}$
form a {\it Jordan chain of length $($or rank\,$)$ $k$} for the operator-valued function $A(\dott)$ at $\lambda_0$ if $\varphi_j\in\cD$, $j=0,\ldots,k-1$, satisfy
\begin{equation}
 \sum_{\ell = 0}^j\frac{1}{\ell!} A^{(\ell)}(\lambda_0)\varphi_{j-\ell}=0,\quad j=0,\dots,k-1,
\end{equation}
and 
\begin{equation} 
\varphi_0\not=0. 
\end{equation} 
The vector $\varphi_0 \in \ker(A(z_0))$ is called an {\it eigenvector} of the operator-valued function $A(\dott)$ 
at the {\it zero} (or, {\it characteristic value}) $\lambda_0$ and the vectors $\varphi_1,\ldots, \varphi_{k-1}$ are said to be {\it generalized eigenvectors} of $A(\dott)$ at $\lambda_0$. \\[1mm]
$(ii)$ The supremum of the length of a chain composed of an eigenvector $\varphi_0 \in \ker(A(\lambda_0))$ and the corresponding generalized eigenvectors of $A(\dott)$ at $\lambda_0$ is called the {\it algebraic multiplicity of $\mathbf \varphi_0$} and denoted by $m_a(\lambda_0; \varphi_0)$. \\[1mm]
$(iii)$ The {\it geometric multiplicity} of the zero $\lambda_0$ of $A(\dott)$, denoted by $m_g(0;A(\lambda_0))$, is defined to be 
\begin{equation} 
m_g(0; A(\lambda_0)) = \dim(\ker(A(\lambda_0))). 
\end{equation} 
$(iv)$ Suppose $\{\varphi_{0,n}\}_{1\leq n \leq N}$, $N \in \bbN \cup \{\infty\}$ represents a basis in 
$\ker(A(\lambda_0))$. Then, the {\it algebraic multiplicity} of the zero $\lambda_0$ of the analytic family $A(\dott)$, denoted by $m_a(\lambda_0; A(\dott))$, is defined via 
\begin{equation}
m_a(\lambda_0; A(\dott)) = \sum_{n=1}^N m_a(\lambda_0; \varphi_{0,n}). 
\end{equation}
\end{definition}
%%%%%%

Assume that the vectors $\{\varphi_0,\ldots,\varphi_{k-1}\} \subset \cD$
form a Jordan chain for the operator-valued function $A(\dott)$ at $\lambda_0$. In contrast to a Jordan chain 
for an eigenvalue of an operator (as in the previous section) here some of the generalized eigenvectors 
$\varphi_1,\ldots, \varphi_{k-1}$ may be zero. It is also important to note that the restricted chain 
$\{\varphi_0,\ldots,\varphi_l\} \subset \cD$, $0\leq l\leq k-1$, is a Jordan chain of length $l+1$ for 
the operator-valued function $A(\dott)$ at $\lambda_0$, and it is also clear that the algebraic multiplicity 
$m_a(\lambda_0; \varphi_0)$ of the eigenvector $\varphi_0$
and the algebraic multiplicity $m_a(\lambda_0; A(\dott))$ of the zero $\lambda_0$ of $A(\dott)$ satisfy
\begin{equation}
m_a(\lambda_0; \varphi_0), m_a(\lambda_0; A(\dott)) \in \bbN \cup \{\infty\}.
\end{equation}
Furthermore, one can show (see, e.g., \cite[p.~57]{Ma88}) that that $m_a(\lambda_0; A(\dott))$ is independent of the basis chosen in $\ker(A(\lambda_0))$, rendering $m_a(\lambda_0; A(\dott))$ well-defined. By definition, 
\begin{equation}
m_g(0; A(\lambda_0)) \leq m_a(\lambda_0; A(\dott)).     \lb{2.6} 
\end{equation}

The next example shows that Definition~\ref{d3.1} is a natural generalization of the concept of Jordan chains of a linear operator (cf.\ \cite[Remark 11.2]{Ma88}).

%%%%%%%
\begin{example} \lb{e3.2} 
Let $A_0$ be a $($possibly unbounded\,$)$ operator $A$ on $\cD=\dom(A)$ in $\cK$ and consider
the special case of the $($linear\,$)$ operator-valued pencil $B(z)=A - z I_{\cK}$, where $z\in\Omega=\bbC$ and $\cD=\dom (B(z))$. Then it follows from Definition~\ref{d3.1} that 
$\{\varphi_0,\ldots,\varphi_{k-1}\} \subset \cD = \dom(A)$
is a Jordan chain of length $k$ for the function $B(\dott)$ at $\lambda_0\in\bbC$ if and only if $\lambda_0$ is an eigenvalue of $A$ with corresponding eigenvector 
$\varphi_0\not=0$ and \eqref{2.1a} holds.
Furthermore, the algebraic multiplicity $m_a(\lambda_0; A)$ of the eigenvalue $\lambda_0$ of the operator $A$ coincides with the algebraic multiplicity $m_a(\lambda_0; B(\dott))$ of the zero $\lambda_0$ of the operator-valued pencil $B(\dott)$, that is, 
\begin{equation}
m_a(\lambda_0; A) = m_a(\lambda_0; B(\dott)).
\end{equation}
\end{example}
%%%%%%%

Assume that $A(\dott)$ is a strongly analytic function defined on some open set $\Omega\subset\bbC$ with 
$\cD=\dom (A(z)) \subseteq \cK$, $z\in\Omega$.
Next, we turn to an equivalent definition of Jordan chains (see, \cite[App.~A]{KM99} and \cite[Sect.~11.2]{Ma88} in the context of bounded analytic families). 
Given $k \in \bbN$ and $\varphi_0,\ldots,\varphi_{k-1} \in \cD$, 
with $\varphi_0$ an eigenvector of $A(\dott)$ corresponding to the zero $\lambda_0 \in \Omega$, 
introduce the vector function
\begin{equation}
\phi(z) = \sum_{j=0}^{k-1} \f{1}{(z- \lambda_0)^{k-j}}\varphi_j.   \lb{2.9}
\end{equation}
One verifies that
\begin{equation}
A(z) \phi(z) \underset{|z - \lambda_0| \downarrow 0}{=} \Oh(1)     \lb{2.10} 
\end{equation}
(in the norm of $\cK$) if and only if $\{\varphi_0,\dots,\varphi_{k-1}\} \subset \cD$ is a Jordan chain of length $k$ for $A(\dott)$. This is a consequence of the fact that
\begin{align}
\begin{split} 
A(z) \phi(z) &= \sum_{j=0}^{\infty} \f{1}{j!} A^{(j)}(\lambda_0) (z-\lambda_0)^j \sum_{\ell=0}^{k-1} \f{1}{(z-\lambda_0)^{k-\ell}} 
\, \varphi_{\ell}   \lb{2.10A} \\
&= \sum_{\ell=0}^{k-1} \f{1}{(z-\lambda_0)^{k-\ell}} \sum_{j=0}^{\ell} \f{1}{j!} A^{(j)}(\lambda_0) \varphi_{\ell-j} + \Oh(1) 
\end{split}
\end{align} 
(here the Cauchy product of two series was invoked to obtain the second equality in \eqref{2.10A}).

Thus, introducing the generalized nullspace associated with the zero $\lambda_0 \in \Omega$ of the strongly 
analytic family $A(\dott)$ via 
\begin{equation}
\cN(\lambda_0; A(\dott)) = \{\text{$\phi$ of the form \eqref{2.9}}\, | \, \phi \text{ satisfies \eqref{2.10}}\},
\end{equation}
the algebraic multiplicity of the zero $\lambda_0$ of $A(\dott)$ equals
\begin{equation}\lb{alg=}
m_a(\lambda_0; A(\dott)) = \dim (\cN(\lambda_0; A(\dott))). 
\end{equation}
The equality \eqref{alg=} is verified in \cite[App.~A]{KM99} in the context of bounded operator functions, but the arguments remain valid in the
slightly more general situation treated here. The main observation to justify \eqref{alg=} is the fact that a Jordan chain 
$\{\varphi_0,\dots,\varphi_{k-1}\} \subset \cD$ of length $k$
for $A(\dott)$ gives rise to the $k$ linearly independent functions
\begin{equation}
 \phi_0(z)=\frac{\varphi_0}{z-\lambda_0},\,\,\phi_1(z)=\frac{\varphi_0}{(z-\lambda_0)^2} 
 + \frac{\varphi_1}{z-\lambda_0},\dots,
 \phi_{k-1}(z) = \sum_{j=0}^{k-1} \f{\varphi_j}{(z- \lambda_0)^{k-j}}
\end{equation}
in $\cN(\lambda_0; A(\dott))$. Note, in particular, that $\varphi_0 \in \ker(A(\lambda_0))$ implies 
$\phi_0 \in \cN(\lambda_0; A(\dott))$, and hence one 
again infers the inequality \eqref{2.6} between geometric and algebraic multiplicities of the zero $\lambda_0$ of $A(\dott)$. 

For future purpose it will be useful to employ the notion of local equivalence of two (strongly analytic) 
operator-valued functions: Let $\cK_j$, $j=1,2$, be Hilbert spaces, let $\Omega \subseteq \bbC$ be open and consider the operator-valued functions 
$z\mapsto A_j(z)$, $j=1,2$, for $z \in \Omega$, with $\dom (A_j(z))=\cD_j$.
Then $A_1(\dott)$ and $A_2(\dott)$ are called {\it locally equivalent} at some point $z_0 \in \Omega$, 
if there exist analytic operator-valued functions $z \mapsto E_1(z) \in \cB(\cK_1,\cK_2)$ and $z \mapsto E_2(z) \in \cB(\cK_2,\cK_1)$ in some open neighborhood 
$\Upsilon(z_0) \subseteq \Omega$ of $z_0$ such that $E_1(z)^{-1} \in \cB(\cK_2,\cK_1)$ and $E_2(z)^{-1} \in \cB(\cK_1,\cK_2)$ and
\begin{equation}
E_2(z)\cD_2=\cD_1\quad\text{and}\quad A_2(z) = E_1(z) A_1(z) E_2(z), \quad z \in \Upsilon(z_0). 
\end{equation}
If, in addition, $A_j(\dott)$, $j=1,2$, are strongly analytic, then one verifies in the same way as in the context of bounded operator functions (see \cite[Proposition A.5.1]{KM99}) that $\lambda_0$ is a zero of $A_1(\dott)$ if and only if $\lambda_0$ is a zero of $A_2(\dott)$, that 
\begin{equation} 
m_a(\lambda_0;A_1(\dott)) < \infty \, \text{ if and only if } \, m_a(\lambda_0;A_2(\dott)) < \infty,     \lb{2.15} 
\end{equation} 
and if one of these numbers in \eqref{2.15} is finite, the algebraic multiplicities of the zero 
$\lambda_0$ of $A_j(\dott)$, $j=1,2$, coincide, and the same applies to the corresponding geometric multiplicities. 

Specializing first to the finite-dimensional situation we now recall the following result (see 
\cite[Theorem~1.1.3]{BGR90}, \cite{GKL78}, \cite[Sect.~1.6]{GLR82}, \cite[Sect.~4.3]{GL09}, \cite[p.~607]{GS71}, 
\cite[Sect.~A.6]{KM99}, \cite[Sect.~7.5]{LT85}):

%%%%%%
\begin{lemma} \lb{lem3.3}
Suppose $N \in \bbN$ and $\lambda_0 \in \bbC$. If $A(\dott) \in \bbC^{N \times N}$ is an $N \times N$ matrix with complex-valued entries analytic at $\lambda_0$, then $A(\dott)$ is locally equivalent to an $N \times N$ diagonal matrix $D(\dott)$ of the $($Smith\,$)$ form
\begin{equation}
D(z) = \diag (\underbrace{0,\dots,0}_{p\, \text{entries}}, (z-\lambda_0)^{\mu_{p+1}}, \dots, (z-\lambda_0)^{\mu_{q}}, 1\dots,1), \quad 0 \leq p \leq q \leq N,
\end{equation} 
for $z$ in an open neighborhood $\Upsilon(z_0) \subseteq \bbC$ of $z_0$, with $\mu_{s} \in \bbN$, 
$s \in \{p+1,\dots, q\}$, where
\begin{equation}
1 \leq \mu_{p+1} \leq \cdots \leq \mu_{q} < \infty. 
\end{equation}
In particular, 
\begin{equation}
m_a (\lambda_0; A(\dott)) = \infty \, \text{ if and only if } \, p \geq 1
\end{equation} 
$($equivalently, if and only if ${\det}_{\bbC^N} (A_N (\dott))\equiv 0$$)$. 
Thus, $m_a (\lambda_0; A(\dott)) < \infty$ if and only if $p=0$, in which case
\begin{equation}
m_a (\lambda_0; A(\dott)) = \sum_{s = 1}^q \mu_s, 
\end{equation}
$($equivalently, $m_a (\lambda_0; A(\dott))$ equals the order of the zero of ${\det}_{\bbC^N} (A(\dott))$ at
$\lambda_0$$)$. 
\end{lemma}
%%%%%%

The following elementary illustrations underscore some aspects of Lemma~\ref{lem3.3}:

%%%%%%
\begin{example} \lb{e3.4}
Consider $N=3$, $\lambda_0 = 0$, and introduce
\begin{equation}
A_1(z) = \diag (z,1,1), \quad z \in \bbC,
\end{equation}
Then $A_1(0) = \diag (0,1,1)$, $A'_1(0) = \diag (1,0,0)$,  $A_1^{(l)}(0) = 0_3$, $l\geq 2$, and 
\begin{equation}
\ker (A_1(0)) = \linspan \{\varphi_0\}, \quad \varphi_0 = (1,0, 0)^{\top},
\end{equation}
and $A_1(0) \varphi_1 + A'_1(0) \varphi_0 = 0$ yields the contradiction $\varphi_0 = 0$, implying the absence of a chain beyond the eigenvector $\varphi_0$. Thus, 
\begin{equation}
m_g(0; A_1(0)) = 1,\quad m_a(0;\varphi_0) = 1,\quad m_a(0; A_1(\dott))=1.    \lb{3.21}
\end{equation} 
\end{example}
%%%%%%

%%%%%%
\begin{example} \lb{e3.5}
Consider $N=3$, $\lambda_0 = 0$, and introduce
\begin{equation}
A_{\infty}(z) = \diag (0,z,1), \quad z \in \bbC,
\end{equation}
Then $A_\infty(0) = \diag (0,0,1)$, $A'_\infty(0) = \diag (0,1,0)$,  $A^{(l)}_\infty(0) = 0_3$, $l\geq 2$, and 
\begin{equation}
 \ker (A_{\infty}(0)) = \linspan \{\varphi_{0,1}, \varphi_{0,2}\}, \quad \, \varphi_{0,1} = (1, 0,0)^{\top}, 
\,\, \varphi_{0,2} = (0,1, 0)^{\top}.
\end{equation}
By inspection, the eigenvector $\varphi_{0,1}$ has an associated chain of length $\infty$ since 
$A_{\infty}(0) \varphi_1 + A'_{\infty}(0) \varphi_{0,1} =0$, and 
$A_{\infty}(0) \varphi_{j+1} + A'_{\infty}(0) \varphi_j =0$, $j \in \bbN$, yields the $($infinite\,$)$ chain, 
\begin{align} 
\begin{split} 
&\varphi_{0,1} = (1, 0,0)^{\top}, \; \varphi_1 = (c_{1,1}, 0,0)^{\top}, 
\; \varphi_2 = (c_{2,1}, 0,0)^{\top}, \dots ,      \\
& \quad \dots, \varphi_j = (c_{j,1}, 0,0)^{\top}, 
\; \varphi_{j+1} = (c_{j+1,1}, 0,0)^{\top}, \dots,  \quad c_{k,1} \in \bbC, \; k \in \bbN.   
\end{split} 
\end{align}
However, as in Example~\ref{e3.4} the equation
$A_{\infty}(0) \varphi_1 + A'_{\infty}(0) \varphi_{0,2} = 0$ yields again the contradiction $\varphi_{0,2} = 0$, implying the absence of a
chain beyond the eigenvector $\varphi_{0,2}$. Thus, 
\begin{equation}
m_g(0; A_{\infty}(0)) = 2, \quad m_a(0; \varphi_{0,1}) = \infty, 
\quad m_a(0; \varphi_{0,2}) = 1, \quad  m_a(0; A_{\infty}(\dott)) = \infty,
\end{equation}
in accordance with ${\det}_{\bbC^3} (A_{\infty}(\dott)) \equiv 0$. 
\end{example}
%%%%%%

%%%%%%
\begin{example} \lb{e3.6}
Consider $N=3$, $\lambda_0 = 0$, and introduce
\begin{equation}
A_2(z) = \diag \big(z,z^2,1\big), \quad z \in \bbC. 
\end{equation}
Thus, $A_2(0) = \diag (0,0,1)$, $A'_2(0) = \diag (1,0,0)$,  $[2 !]^{-1}A''_2(0) = \diag (0,1,0)$, $A_2^{(l)}(0) = 0_3$, $l\geq 3$, and 
\begin{equation}
\ker (A_2(0)) = \linspan \{\varphi_{0,1}, \varphi_{0,2}\}, \quad \, \varphi_{0,1} = (1,0, 0)^{\top}, \,\, \varphi_{0,2} = (0,1,0)^{\top}.
\end{equation}
As in the previous examples the equation $A_2(0) \varphi_1 + A_2'(0) \varphi_{0,1} = 0$ leads to the contradiction $\varphi_{0,1} = 0$.
Hence there is no chain beyond the eigenvector $\varphi_{0,1}$.
Similarly, $A_2(0) \varphi_1 + A_2'(0) \varphi_{0,2} = 0$, with $\varphi_1= (c_{1,1}, c_{1,2}, c_{1,3})^{\top}$ implies $c_{1,3} = 0$, 
and thus yields the chain $($with $c_{1,1}, c_{1,2} \in \bbC$$)$
\begin{equation} 
\varphi_{0,2} = (0,1,0)^{\top}, \; \varphi_1 = (c_{1,1}, c_{1,2},0)^{\top}.  
\end{equation} 
Next, studying $A_2(0) \varphi_2 + A_2'(0) \varphi_1 + [2!]^{-1} A_2''(0)\varphi_{0,2} = 0$, with some  
$\varphi_2$ yields the contradiction 
$\varphi_{0,2}= 0$. Thus, there exists no generalized eigenvector $\varphi_2$ and hence no chain of length $\geq 3$. Summing up, 
\begin{equation}
m_g(0; A_2(0)) = 2, 
 \quad m_a(0;\varphi_{0,1}) = 1, \quad m_a(0;\varphi_{0,2}) = 2,
\quad  m_a(0; A_2(\dott)) = 3.      \lb{3.29}
\end{equation} 
\end{example}
%%%%%%

%%%%%%
\begin{example} \lb{e3.7}Consider $N=3$, $\lambda_0 = 0$, and introduce
\begin{equation}
A_k(z) = \diag \big(z,z^k,1\big), \quad z \in \bbC,\,\, k\geq 3.
\end{equation}
Then $A_k(0) = \diag (0,0,1)$, $A'_k(0) = \diag (1,0,0)$, $[k !]^{-1}A^{(k)}_k(0) = \diag (0,1,0)$, 
$A^{(l)}_k(0) = 0_3$ for $l\geq 2$, $l\not=k$, and 
\begin{equation}
\ker (A_k(0)) = \linspan \{\varphi_{0,1}, \varphi_{0,2}\}, \quad \, \varphi_{0,1} = (1,0, 0)^{\top}, \,\, \varphi_{0,2} = (0,1,0)^{\top}.
\end{equation}
In the same way as in the previous examples one verifies that there is no chain beyond the eigenvector $\varphi_{0,1}$
and a chain of maximal length $k$ beyond the eigenvector $\varphi_{0,2}$. This leads to
\begin{equation}
m_g(0; A_k(0)) = 2, 
 \quad m_a(0;\varphi_{0,1}) = 1, \quad m_a(0;\varphi_{0,2}) = k,
\quad  m_a(0; A_k(\dott)) = k+1.     \lb{3.32} 
\end{equation} 
\end{example}
%%%%%%

In the general finite-dimensional situation one obtains the following:

%%%%%%
\begin{example} \lb{e3.8}Consider $N \in \bbN$, $\lambda_0 = 0$, and introduce 
\begin{equation}
A_{k_1, \dots, k_N}(z) = \diag \big(z^{k_1},\dots,z^{k_N},1\big), \quad 1 \leq k_j \leq k_{j+1}, \; 1 \leq j \leq N-1. 
\end{equation}
Then $A_{k_1, \dots, k_N}(0) = \diag (0,\dots,0,1)$,
\begin{align}
\begin{split}
& \ker(A_{k_1,\dots,k_N}(0)) = \linspan \{\varphi_{0,j}\}_{1 \leq j \leq N},    \\
& \, \varphi_{0,j} = (0,\dots,0,\underbrace{1}_{j},0,\dots0)^{\top}, \quad  1 \leq j \leq N.
\end{split} 
\end{align}
Each term $z^{k_j}$ in $A_{k_1, \dots, k_N}(z)$ then leads to a Jordan chain of $($maximal\,$)$ length $k_j$ since
\begin{align}
 [k_j!]^{-1}A^{(k_j)}_{k_1,\dots,k_N}(0) & = \diag (0,\dots,0,\underbrace{1,\dots,\underbrace{1}_{j},\dots,1},0,\dots,0),   \lb{2.39} \\
 [\ell !]^{-1}A^{(\ell)}_{k_1,\dots,k_N}(0) &= \diag (\ast,\dots,\ast,\underbrace{0,\dots,\underbrace{0}_{j},\dots,0},0,\dots,0), \quad 1 \leq \ell \leq k_j -1,   \lb{2.40}
\end{align}
where the larger underbraced part $\underbrace{\dots}$ characterizes all those $j' \in \{1,\dots,N\}$ such that 
$k_{j'} = k_{j}$ and $\ast$ stands for $0$ or $1$. Because of \eqref{2.39} and \eqref{2.40}, the equation 
\begin{equation}
A_{k_1,\dots,k_N}(0) \varphi_{k_j} + A_{k_1,\dots,k_N}'(0) \varphi_{k_j - 1} + \dots + [k_j!]^{-1} A_{k_1,\dots,k_N}^{(k_j)}(0) \varphi_{0,j} =0 
\end{equation} 
implies the contradiction $\varphi_{0,j} = 0$ and hence no chain of length $k_j + 1$ 
$($containing $\varphi_{k_j}$$)$ exists. The explicit form of $A^{(k_j)}_{k_1,\dots,k_N}(0)$ in \eqref{2.39} then shows that 
\begin{equation}
\text{the Jordan chain $\varphi_{0,j}, \varphi_1, \dots , \varphi_{k_j - 1}$ of $($maximal\,$)$ length $k_j$ exists.}
\end{equation}
Thus,
\begin{equation}
m_g(0;A_{k_1,\dots,k_N}(\dott)) = N, \quad m_a(0;\varphi_{0,j}) = k_j, \quad m_a(0; A_{k_1,\dots,k_N}(\dott)) = \sum_{j=1}^N k_j.    \lb{3.39} 
\end{equation}
\end{example}
%%%%%%

To extend the finite-dimensional situation described in Lemma~\ref{lem3.3} to the infinite-dimensional case, we next recall the notion of a zero of finite-type of a bounded analytic function following \cite[Sects.\ XI.8, XI.9]{GGK90}, \cite{GS71} (see also \cite{GHN15}, \cite[Ch.\ 4]{GL09}). 

%%%%%%
\begin{definition} \lb{d3.9}
Let $\Omega \subseteq \bbC$ be open, $\lambda_0 \in \Omega$, and suppose that 
$A:\Omega \to \cB(\cH)$ is analytic on $\Omega$. Then $\lambda_0$ is called a {\it zero of finite-type 
of $A(\dott)$} if $A(\lambda_0)$ is a Fredholm operator, $\ker(A(\lambda_0)) \neq \{0\}$, and $A(\dott)$ is 
boundedly invertible on $D(\lambda_0; \varepsilon_0) \backslash \{z_0\}$, for some sufficiently small 
$\varepsilon_0 > 0$. 
\end{definition}
%%%%%%

In particular, the hypotheses imposed in Definition~\ref{d3.9} imply that $A(z)$ is Fredholm for all 
$z \in D(\lambda_0; \varepsilon_0)$. 
In the context of bounded analytic operator-valued functions we also recall that the notions of weakly analytic, strongly analytic, and norm analytic families are all equivalent.

Combining various results in \cite{GS71} (in particular, p.~605, eqs.(1.1)--(1.3), (3.1)--(3.3), Lemma~2.1, Theorems~3.1, 3.2, and the last paragraph in the proof of Lemma~2.1 on p.~613) and \cite[Theorem\ XI.8.1]{GGK90} (see also \cite{GKL78}), one then obtains the following infinite-dimensional 
analog of Lemma~\ref{lem3.3}:

%%%%%%
\begin{theorem} \lb{t3.10}
Assume that $A:\Omega \to \cB(\cH)$ is analytic on $\Omega$ and that
$\lambda_0 \in \Omega$ is a zero of finite-type of $A(\dott)$. 
Then
\begin{equation}
\ind (A(\lambda_0)) =\dim(\ker(A(\lambda_0))) - \dim(\ker(A(\lambda_0)^*))= 0,  
\end{equation}
and there exist $\varepsilon > 0$, analytic and boundedly invertible 
operator-valued functions $E_j: \Omega \to \cB(\cH)$, $j=1,2$, and mutually disjoint orthogonal projections $P_k$, $k=0,\ldots,r$, in $\cH$ 
with 
\begin{equation}
 \dim (\ran(P_j)) = 1, \quad 1 \leq j \leq r,\quad \bigoplus_{j=0}^r P_j = I_{\cH},
\end{equation}
and uniquely determined $\rho_1 \leq \rho_2 \leq \dots \leq \rho_r$, $\rho_j \in \bbN$, $j=1\ldots,r$,
such that 
\begin{equation}\lb{2.27}
A(z) = E_1(z) D(z) E_2(z), \quad z \in D(\lambda_0; \varepsilon), 
\end{equation}
where $D(\dott)$ admits the diagonal block operator form
\begin{equation}\lb{2.28}
D(z)=\left[\begin{matrix} (z - \lambda_0)^{\rho_1} & & & & \\ & (z - \lambda_0)^{\rho_2} & & & & 
\\ & & \ddots & & \\ & & & (z - \lambda_0)^{\rho_r} & 
            \\ & & & & I_{\ran (P_0)} \end{matrix}\right]
\end{equation}
with respect to the decomposition 
\begin{equation}
 \cH=\ran (P_1)\oplus\ran (P_2)\oplus\dots\oplus \ran (P_r)\oplus \ran (P_0).
\end{equation}
The geometric multiplicity $m_g(0; A(z_0))$ and the algebraic multiplicity $m_a(\lambda_0; A(\dott))$ of the zero of 
$A(\dott)$ at $\lambda_0$ are given by
\begin{equation} 
m_g(0; A(\lambda_0)) = \dim\big(\ran\big(I_{\cH} - P_0\big)\big) =r, \, \text{ and } \, m_a(\lambda_0; A(\dott)) = \sum_{j=1}^r \rho_j.
\end{equation}
\end{theorem}
%%%%%%

Without going into further details we emphasize that \cite{GS71} actually focuses on meromorphic operator-valued functions, not just the special analytic case.

%%%%%%%%%%%%%%%%%%%%%%%%%%%%%%%%%%%%%%
%%%%%%%%%%%%%%%%%%%%%%%%%%%%%%%%%%%%%%
\section{The Generalized Birman--Schwinger Principle and Jordan Chains} \lb{s4}
%%%%%%%%%%%%%%%%%%%%%%%%%%%%%%%%%%%%%%
%%%%%%%%%%%%%%%%%%%%%%%%%%%%%%%%%%%%%%

This section is devoted to the generalized Birman--Schwinger principle in connection with a pair of operators 
$(H_0, H)$ in a separable, complex Hilbert space $\cH$ that satisfy the following hypothesis.

%%%%%%
\begin{hypothesis}\lb{h4.1}
Let $H_0$ be a closed operator in $\cH$ with $\rho(H_0)\not=\emptyset$, and assume that 
$V_1,V_2$ are $($possibly unbounded\,$)$ operators mapping from $\cH$ into an auxiliary Hilbert 
space $\cK$ such that $V_1$ is closed, and 
\begin{equation}
\dom (H_0) \subseteq \dom (V_2^*V_1),
\quad \text{ and } \, \overline{\dom (V_2)}=\cH.
\end{equation}
One then introduces 
\begin{equation}\lb{h}
 H=H_0+V_2^*V_1, \quad \dom(H) = \dom(H_0).
\end{equation}
\end{hypothesis} 
%%%%%%

One notes that the operator $H$ in Hypothesis~\ref{h4.1} is not necessarily closed and 
that the case $\rho(H)=\emptyset$ is not excluded.
Since $V_1$ is assumed to be closed it follows from the closed graph theorem that 
$V_1(H_0-z I_{\cH})^{-1}\in\cB(\cH,\cK)$ holds for all $z\in\rho(H_0)$ and hence one can use 
$\cD = \dom(V_2^*) \subseteq \cK$ for the special operator-valued function
\begin{equation}
\rho(H_0) \ni z \mapsto  I_{\cK} + V_1(H_0 - z I_{\cH})^{-1} V_2^*.
\end{equation}  
In the following it will also be used that
\begin{equation}\lb{resodiff}
 \frac{d^{\ell}}{dz^{\ell}}(H_0-z I_{\cH})^{-1}= \ell!\,(H_0-z I_{\cH})^{-(\ell +1)}, \quad \ell\in\bbN, 
 \; z\in\rho(H_0), 
\end{equation} 
and that 
\begin{equation}\lb{taylor}
 (H_0-z I_{\cH})^{-1}=\sum_{s=0}^\infty (H_0-z_0 I_{\cH})^{-(s+1)}(z - z_0)^s
\end{equation}
for all $z $ in a neighbourhood of $z_0\in\rho(H_0)$.

%%%%%%
\begin{lemma} \lb{l4.2}
Assume Hypothesis~\ref{h4.1}. If $\varphi \in \dom (V_2^*)$, then the map
\begin{equation}\lb{lokm}
\rho(H_0) \ni z\mapsto  \big[I_{\cK} + V_1(H_0-z I_{\cH})^{-1}V_2^*\big] \varphi \in \cK
\end{equation}
is analytic and 
\begin{equation}\lb{derivatives}
 \frac{d^{\ell}}{dz^{\ell}} V_1(H_0-z I_{\cH})^{-1}V_2^*  \varphi=\ell! V_1(H_0-z I_{\cH})^{-(\ell +1)}V_2^* \varphi,\quad \ell\in\bbN.
\end{equation}
\end{lemma}
%%%%%%
\begin{proof}
In fact, since $V_1(H_0-z_0 I_{\cH})^{-1}\in\cB(\cH,\cK)$ for $z_0\in\rho(H_0)$ one has
\begin{equation}
\begin{split}
&\lim_{z \rightarrow z_0} \frac{1}{z - z_0} 
\big[V_1(H_0-z I_{\cH})^{-1}V_2^* - V_1(H_0-z_0 I_{\cH})^{-1}V_2^*\big]\varphi\\
&\quad =  \lim_{z \rightarrow z_0} V_1(H_0-z_0 I_{\cH})^{-1}(H_0-z I_{\cH})^{-1}V_2^*\varphi\\
&\quad =  V_1(H_0-z_0 I_{\cH})^{-2}V_2^*\varphi, \quad \varphi\in\dom (V_2^*). 
\end{split} 
\end{equation}
Hence the function \eqref{lokm} is analytic from $\rho(H_0)$ into $\cK$.
In the same manner, making use of \eqref{resodiff} and $V_1(H_0-z_0 I_{\cH})^{-1}\in\cB(\cH,\cK)$, one verifies \eqref{derivatives} by induction. 
\end{proof}
%%%%%%

With $\cD=\dom (V_2^*)$ and $\Omega=\rho(H_0)$, the next result follows immediately 
from Definition~\ref{d3.1} and Lemma~\ref{l4.2}.

%%%%%%
\begin{corollary} \lb{c4.3}
Assume Hypothesis~\ref{h4.1} and let 
$z_0\in\rho(H_0)$. Then the collection 
$\{\varphi_0,\ldots,\varphi_{k-1}\} \subset \cD = \dom(V_2^*)$  
form a Jordan chain of length $k \in \bbN$ for the function 
\begin{equation}\lb{bsfunction202}
\rho(H_0) \ni z\mapsto  I_{\cK} + V_1(H_0-z I_{\cH})^{-1}V_2^* 
\end{equation}
at $z_0\in\rho(H_0)$ if and only if $\varphi_0 \neq 0$
and for all $j \in \{ 0,\ldots,k-1\} $ one has 
$\varphi_j \in \dom (V_2^*)$ and 
\begin{equation}\lb{assu}
 \sum_{\ell = 0}^j V_1(H_0-z_0 I_{\cH})^{-(\ell +1)}V_2^*\varphi_{j-\ell}=-\varphi_j, \quad 
 j = 0,\dots,k-1.
\end{equation}
\end{corollary}
%%%%%%

We continue with the following auxiliary result, which can be viewed as a variant of \cite[Lemma 4.5]{BE19} and goes back to more abstract
considerations in \cite[Sect.~7.4.4]{DM17}.
 
%%%%%%
\begin{lemma}\lb{l4.4}
Let $H_0$ and $H=H_0+V_2^*V_1$ be as in Hypothesis~\ref{h4.1}, let $z_0\in\rho(H_0)$
and let $\{f_0,\dots, f_{k-1}\} \subset \dom(H)$, $k \in \bbN$, be a Jordan chain of length $k$ for $H$ at $z_0$. 
Then for all $j \in \{ 1,\dots,k\} $, 
 \begin{equation}
  -V_1(H_0-z_0 I_{\cH})^{-1}f_{j-1} = 
  \sum_{\ell = 1}^j V_1 (H_0-z_0 I_{\cH})^{-(\ell +1)}V_2^*V_1 f_{j-\ell} . 
 \end{equation}
\end{lemma}
%%%%%%
\begin{proof}
We shall show by induction that 
 \begin{align}\lb{big}
 \begin{split}
& -V_1(H_0-z I_{\cH})^{-1}f_{j-1}
 =\sum_{\ell = 1}^j \frac{1}{(z - z_0)^\ell}V_1     \\
 & \quad \times \bigg((H_0-z I_{\cH})^{-1} 
  -\sum_{s=0}^{\ell-1}(H_0-z_0 I_{\cH})^{-(s+1)}(z - z_0)^s\bigg)V_2^*V_1 f_{j-\ell}
 \end{split}
 \end{align}
for all $j \in \{ 1,\dots,k\} $ and 
$z  \in\rho(H_0) \backslash  \{ z_0 \} $. The assertion of the 
 lemma then follows by taking the limit
 $z \rightarrow z_0$ in \eqref{big}. 
Indeed, using 
 $V_1(H_0-z_0 I_{\cH})^{-1}\in\cB(\cH,\cK)$ one obtains for the limit on the left-hand side of \eqref{big} 
 \begin{align}
 \begin{split} 
& \lim_{z \to z_0}
   V_1(H_0-z I_{\cH})^{-1}f_{j-1} - V_1(H_0-z_0 I_{\cH})^{-1}f_{j-1}    \\
& \quad = \lim_{z \to z_0}
     V_1(H_0-z_0 I_{\cH})^{-1}(z - z_0)(H_0-z I_{\cH})^{-1}f_{j-1} = 0. \lb{limarg} 
\end{split}
\end{align}
On the other hand, using the Taylor expansion \eqref{taylor} of the resolvent 
$z\mapsto (H_0-z I_{\cH})^{-1}$ in a neighbourhood of $z_0\in\rho(H_0)$ shows that 
\begin{align}\lb{limarg2}
&\frac{1}{(z - z_0)^\ell}V_1\bigg((H_0-z I_{\cH})^{-1} 
-\sum_{s=0}^{\ell-1}(H_0-z_0 I_{\cH})^{-(s+1)}(z - z_0)^s\bigg)V_2^*V_1 f_{j-\ell}    \no \\
&\quad =\frac{1}{(z - z_0)^\ell}V_1\bigg((H_0-z_0 I_{\cH})^{-(\ell +1)}(z - z_0)^\ell   \no \\
& \hspace*{28mm} + \sum_{s= \ell+1}^{\infty}(H_0-z_0 I_{\cH})^{-(s+1)}(z - z_0)^s\bigg)V_2^*V_1 f_{j-\ell}  \no \\
&\quad =V_1 (H_0-z_0 I_{\cH})^{-(\ell +1)}V_2^*V_1 f_{j-\ell}   \no  \\ 
& \qquad + V_1 (H_0-z_0 I_{\cH})^{-1}
\sum_{s= \ell+1}^{\infty}(H_0-z_0 I_{\cH})^{-s}(z - z_0)^{s-\ell}V_2^*V_1 f_{j-\ell},   
 \end{align}
and using again 
 $V_1(H_0-z_0 I_{\cH})^{-1}\in\cB(\cH,\cK)$
 it follows that in the limit $z \rightarrow z_0$
 the right-hand side of \eqref{big} 
 tends to 
 \begin{equation}
 \sum_{\ell = 1}^j V_1 (H_0-z_0 I_{\cH})^{-(\ell +1)}V_2^*V_1 f_{j-\ell} .
 \end{equation}
Therefore, taking the limit
 $z \rightarrow z_0$ in \eqref{big} implies the assertion of the lemma.

Next, we prove \eqref{big} for $j=1$.
By assumption we have  $(H-z_0 I_{\cH})f_0 = 0$ and hence $V_2^*V_1f_0=-(H_0-z_0 I_{\cH}) f_0$. 
Therefore, 
 \begin{align}
 & \frac{1}{z - z_0}V_1\big[(H_0-z I_{\cH})^{-1}  - (H_0-z_0 I_{\cH})^{-1} \big] V_2^*V_1 f_0  \no \\
 & \quad =V_1 (H_0-z I_{\cH})^{-1} (H_0-z_0 I_{\cH})^{-1}  V_2^*V_1 f_0  \no \\
 & \quad =-V_1 (H_0-z I_{\cH})^{-1}  f_0, \quad z  \in\rho(H_0) \backslash  \{ z_0 \},  
 \end{align}
which gives the desired formula \eqref{big} for $j=1$.

Next, let $m \in \{ 1,\ldots,k-1\} $ and assume that the formula \eqref{big} holds for $j=m$
 and all $z  \in\rho(H_0) \backslash  \{ z_0 \} $.
Taking the limit $z \rightarrow z_0$
one deduces that 
 \begin{equation}\lb{ok}
  -V_1(H_0-z_0 I_{\cH})^{-1}f_{m-1}
=\sum_{\ell = 1}^m V_1 (H_0-z_0 I_{\cH})^{-(\ell +1)}V_2^*V_1 f_{m- \ell } ,
 \end{equation}
by arguing as in \eqref{limarg} and \eqref{limarg2}.
 
Let $z  \in\rho(H_0) \backslash  \{ z_0 \}$.
We next prove formula \eqref{big} for $j=m+1$.
Starting with the right-hand side of \eqref{big} a computation yields
 \begin{align}
  &\sum_{\ell = 1}^{m+1} \frac{1}{(z - z_0)^\ell}V_1\bigg((H_0-z I_{\cH})^{-1} 
  -\sum_{s=0}^{\ell-1}(H_0-z_0 I_{\cH})^{-(s+1)}(z - z_0)^s\bigg)V_2^*V_1 f_{m+1-\ell}    \no \\
  &\quad =\sum_{\ell = 2}^{m+1} \frac{1}{(z - z_0)^\ell}V_1\bigg((H_0-z I_{\cH})^{-1} 
  -\sum_{s=0}^{\ell-1}(H_0-z_0 I_{\cH})^{-(s+1)}(z - z_0)^s\bigg)    \no \\
  & \hspace*{33mm} \times V_2^*V_1 f_{m+1-\ell}    \no \\
  &\qquad + \frac{1}{z - z_0}V_1\big[(H_0-z I_{\cH})^{-1} 
  - (H_0-z_0 I_{\cH})^{-1}\big] V_2^*V_1 f_m   \no \\
  &\quad =\sum_{\ell = 1}^{m} \frac{1}{(z - z_0)^{\ell+1}}V_1\bigg((H_0-z I_{\cH})^{-1} 
  -\sum_{s=0}^{\ell}(H_0-z_0 I_{\cH})^{-(s+1)}(z - z_0)^s\bigg)     \no \\
 & \hspace*{35mm} \times V_2^*V_1 f_{m- \ell }   \no \\
  &\qquad + \frac{1}{z - z_0}V_1\big[(H_0-z I_{\cH})^{-1} 
  - (H_0-z_0 I_{\cH})^{-1}\big]V_2^*V_1 f_m     \no \\
  &\quad =\frac{1}{z - z_0}
      \sum_{\ell = 1}^{m} \frac{1}{(z - z_0)^{\ell}}V_1\bigg((H_0-z I_{\cH})^{-1} 
  -\sum_{s=0}^{\ell-1}(H_0-z_0 I_{\cH})^{-(s+1)}(z - z_0)^s\bigg)    \no \\
  &  \hspace*{42mm} \times V_2^*V_1 f_{m- \ell }    \no \\
  &\qquad  -\frac{1}{z - z_0}\sum_{\ell = 1}^m V_1(H_0-z_0 I_{\cH})^{-(\ell +1)}V_2^*V_1 f_{m- \ell } \no \\
  &\qquad + \frac{1}{z - z_0}V_1
     \big[(H_0-z I_{\cH})^{-1} - (H_0-z_0 I_{\cH})^{-1}\big] V_2^*V_1 f_m.    \lb{3.35} 
 \end{align}
At this point we use assumption \eqref{big} for $j=m$ for the first term on the right-hand side,  
\eqref{ok} for the second term on the right-hand side of \eqref{3.35}, 
and $(H-z_0 I_{\cH})f_m=f_{m-1}$ to conclude that  formula \eqref{3.35} becomes
\begin{align}
 &\sum_{\ell = 1}^{m+1} \frac{1}{(z - z_0)^\ell}V_1\bigg((H_0-z I_{\cH})^{-1}    \no \\
& \hspace*{30mm} 
 -\sum_{s=0}^{\ell-1}(H_0-z_0 I_{\cH})^{-(s+1)}(z - z_0)^s\bigg)V_2^*V_1 f_{m+1-\ell}\no \\
  &\quad = \frac{1}{z - z_0} -V_1(H_0-z I_{\cH})^{-1}f_{m-1}    \no \\
 & \qquad +\frac{1}{z - z_0} V_1(H_0-z_0 I_{\cH})^{-1}f_{m-1}  \no \\
  &\qquad + \frac{1}{z - z_0} V_1\big[(H_0-z I_{\cH})^{-1} - (H_0-z_0 I_{\cH})^{-1}\big]V_2^*V_1 f_m  \no \\
  &\quad=\frac{1}{z - z_0} V_1\big[(H_0-z I_{\cH})^{-1} - (H_0-z_0 I_{\cH})^{-1}\big](V_2^*V_1 f_m-f_{m-1})  \no \\
  &\quad= V_1 (H_0-z I_{\cH})^{-1}  
  (H_0-z_0 I_{\cH})^{-1}(V_2^*V_1 f_m-(H-z_0 I_{\cH})f_m)
  \no \\
  &\quad=- V_1 (H_0-z I_{\cH})^{-1}  (H_0-z_0 I_{\cH})^{-1}(H_0-z_0 I_{\cH})f_m  \no \\
  &\quad=- V_1 (H_0-z I_{\cH})^{-1}  f_m,
\end{align}
which is \eqref{big} for $j=m+1$. 
\end{proof}
%%%%%%

The generalized Birman--Schwinger principal then reads as follows:

%%%%%%
\begin{theorem}\lb{t4.5}
Let $H_0$ and $H=H_0+V_2^*V_1$ be as in Hypothesis~\ref{h4.1}, assume 
$z_0\in\rho(H_0)$, and consider the map 
\begin{equation}\lb{bsfunction}
\rho(H_0) \ni z\mapsto  I_{\cK} + V_1(H_0-z I_{\cH})^{-1}V_2^*. 
\end{equation} 
Then the following items $(i)$ and $(ii)$ hold: \\[1mm] 
$(i)$ Let $\{f_0,\dots,f_{k-1}\} \subset \dom(H)$, $k \in \bbN$, be a Jordan chain of length $k$ for $H$ at $z_0$.
Define 
\begin{equation} 
\varphi_j = V_1 f_j
\end{equation} 
for all $j \in \{ 0,\ldots,k-1\} $.
Then $\{\varphi_0,\dots, \varphi_{k-1}\} \subset \dom(V_2^*)$ is a Jordan chain of length $k$ 
for the function \eqref{bsfunction} at $z_0$. \\[1mm] 
$(ii)$ Let $\{\varphi_0,\dots,\varphi_{k-1}\} \subset \dom(V_2^*)$, $k \in \bbN$, be a Jordan chain for the function \eqref{bsfunction}
at $z_0$.
Define 
\begin{equation}\lb{f0}
  f_0=-(H_0-z_0 I_{\cH})^{-1}V_2^*\varphi_0
\end{equation}
and inductively define 
\begin{equation} 
f_j = - (H_0-z_0 I_{\cH})^{-1}
    \bigg( f_{j-1} - V_2^*\sum_{\ell = 0}^j V_1(H_0-z_0 I_{\cH})^{-(\ell +1)}V_2^*\varphi_{j - \ell }\bigg)  \lb{3.13} 
\end{equation} 
for all $j \in \{ 1,\ldots,k-1\} $.
Then $\{f_0,\dots,f_{k-1}\} \in \dom(H)$
is a Jordan chain of length $k$ for $H$ at $z_0$ with $V_1f_j = \varphi_j$ for all 
$j \in \{ 0,\ldots,k-1\} $.
\end{theorem}
%%%%%% 
\begin{proof} 
$(i)$ Assume that $\{f_0,\dots,f_{k-1}\} \subset \dom(H)$ is a Jordan chain of length $k$ for $H$ at $z_0$.
If $j \in \{ 0,\ldots,k-1\} $, then $f_j \in \dom (H) \subseteq \dom (V_2^*V_1) \subseteq \dom (V_1)$.
For all $j \in \{ 0,\ldots,k-1\}$, define $\varphi_j = V_1 f_j$.
Then $\varphi_0 \neq 0$, as otherwise 
$(H_0-z_0 I_{\cH})f_0=-V_2^*V_1 f_0=0$ and therefore $f_0 = 0$, since $z_0 \in \rho(H_0)$.
We shall show that 
\begin{equation}\lb{assuxx}
\sum_{\ell = 0}^j V_1(H_0-z_0 I_{\cH})^{-(\ell +1)}V_2^*V_1 f_{j-\ell}=-V_1 f_j
\end{equation}
for all $j \in \{ 0,\ldots,k-1\} $, which proves item $(i)$ by Corollary~\ref{c4.3}.

If $j=0$, then $(H-z_0 I_{\cH})f_0=0$ and hence $V_2^*V_1 f_0=-(H_0-z_0 I_{\cH})f_0$.
This implies
\begin{equation}
 V_1(H_0-z_0 I_{\cH})^{-1}V_2^*V_1f_0=-V_1 f_0,
\end{equation}
which is \eqref{assuxx} for $j=0$.

Let $j \in \{ 1,\ldots,k-1\} $.
Making use of Lemma~\ref{l4.4} one obtains
\begin{align}
 & -V_1(H_0-z_0 I_{\cH})^{-1}V_2^*V_1 f_j-V_1 f_j     \no \\
 & \quad = -V_1(H_0-z_0 I_{\cH})^{-1} (V_2^*V_1+H_0-z_0 I_{\cH})  f_j \no \\
 & \quad = -V_1(H_0-z_0 I_{\cH})^{-1} (H-z_0 I_{\cH}) f_j   \no \\
 & \quad = -V_1(H_0-z_0 I_{\cH})^{-1} f_{j-1}   \no \\
 & \quad =\sum_{\ell = 1}^j V_1 (H_0-z_0 I_{\cH})^{-(\ell +1)}V_2^*V_1 f_{j-\ell} .
\end{align}
Consequently
\begin{equation}
\begin{split}
 -V_1 f_j&=V_1(H_0-z_0 I_{\cH})^{-1}V_2^*V_1 f_j +\sum_{\ell = 1}^j V_1 (H_0-z_0 I_{\cH})^{-(\ell +1)}V_2^*V_1 f_{j-\ell}\\
         &=\sum_{\ell = 0}^j V_1 (H_0-z_0 I_{\cH})^{-(\ell +1)}V_2^*V_1 f_{j-\ell},
 \end{split}
 \end{equation}
and hence \eqref{assuxx} holds.

$(ii)$ Assume that $\{\varphi_0,\dots,\varphi_{k-1}\} \subset \dom(V_2^*)$ is a Jordan chain of length $k$ 
for the function \eqref{bsfunction} at $z_0\in\rho(H_0)$.
Then \eqref{assu} in Corollary~\ref{c4.3} holds for all $j \in \{ 0,\ldots,k-1\} $  and
 $\varphi_0\not= 0$.
Define $f_0,\ldots,f_{k-1}$ as in \eqref{3.13}. We now proceed by induction. 

The assumption \eqref{assu} for $j=0$ gives
$V_1(H_0-z_0 I_{\cH})^{-1}V_2^*\varphi_0=-\varphi_0$ and with $f_0$ as in \eqref{f0} one obtains
\begin{equation}
\begin{split}
 (H-z_0 I_{\cH})f_0&=(H_0-z_0 I_{\cH})f_0+V_2^*V_1 f_0\\
 &=-V_2^*\varphi_0 - V_2^*V_1(H_0-z_0 I_{\cH})^{-1}V_2^*\varphi_0\\
 &=0.
\end{split}
\end{equation}
Moreover, $V_1f_0=-V_1(H_0-z_0 I_{\cH})^{-1}V_2^*\varphi_0=\varphi_0$.
Since $\varphi_0\not=0$ it also follows that $f_0\not=0$.

Let $m \in \{ 1,\ldots,k-1\} $ and assume that $\{f_0,\dots,f_{m-1}\}$ is a Jordan chain of length $m$ for
 $H$ at $z_0$ and $\varphi_j=V_1 f_j$ for all $j \in \{ 0,\ldots,m-1 \} $.
It follows from Lemma~\ref{l4.4} that 
\begin{equation}\lb{useit}
 -V_1(H_0-z_0 I_{\cH})^{-1}f_{m-1}
 = \sum_{\ell = 1}^m V_1 (H_0-z_0 I_{\cH})^{-(\ell +1)}V_2^*V_1 f_{m- \ell }
 .
\end{equation}
We shall show that $(H-z_0 I_{\cH})f_m=f_{m-1}$ and $V_1 f_m = \varphi_m$.
For convenience we next set 
\begin{equation}
\wti f_m=-\sum_{\ell = 0}^m (H_0-z_0 I_{\cH})^{-(\ell +1)}V_2^*\varphi_{m- \ell}.
\end{equation}
Then 
\begin{align}
& \wti f_m-(H_0-z_0 I_{\cH})^{-1}\big[(H-z_0 I_{\cH})\wti f_m-f_{m-1} \big]    \no \\
& \quad = (H_0-z_0 I_{\cH})^{-1} \big[(H_0-z_0 I_{\cH})\wti f_m - (H - z_0 I_{\cH})\wti f_m - f_{m-1} \big]  
\no \\
& \quad = (H_0-z_0 I_{\cH})^{-1} \big[- V_2^* V_1\wti f_m - f_{m-1} \big] \no  \\
& \quad = - (H_0-z_0 I_{\cH})^{-1} \big[f_{m-1} + V_2^* V_1\wti f_m \big]    \no \\
& \quad = f_m \lb{fm}
\end{align}
by the definition of $f_m$. It follows from \eqref{assu} with $j=m$ that $V_1\wti f_m=\varphi_m$. Next, 
\begin{align}
& -V_1(H_0-z_0 I_{\cH})^{-1}(H-z_0 I_{\cH})\wti f_m \no \\
& \quad =-V_1(H_0-z_0 I_{\cH})^{-1}[V_2^*V_1+H_0-z_0 I_{\cH}]\wti f_m  \no \\
 & \quad =-V_1(H_0-z_0 I_{\cH})^{-1}V_2^*V_1\wti f_m-V_1\wti f_m  \no \\
 & \quad = -V_1(H_0-z_0 I_{\cH})^{-1}V_2^*\varphi_m-\varphi_m ,  
\end{align}
and with the help of \eqref{assu}  with $j=m$, the induction hypothesis,  
and \eqref{useit} one obtains
 \begin{align}
 & -V_1(H_0-z_0 I_{\cH})^{-1}(H-z_0 I_{\cH})\wti f_m    \no \\
 & \quad = \sum_{\ell = 1}^m V_1(H_0-z_0 I_{\cH})^{-(\ell +1)}V_2^*\varphi_{m- \ell } \no \\
 & \quad = \sum_{\ell = 1}^m V_1(H_0-z_0 I_{\cH})^{-(\ell +1)}V_2^*V_1 f_{m- \ell }   \no \\
 & \quad = -V_1(H_0-z_0 I_{\cH})^{-1}f_{m-1}.
 \end{align}
Hence
\begin{equation}\lb{opop}
 V_1(H_0-z_0 I_{\cH})^{-1}\big[(H-z_0 I_{\cH})\wti f_m-f_{m-1} \big]
=0.
\end{equation}
Therefore also 
\begin{equation}
 V_2^*V_1(H_0-z_0 I_{\cH})^{-1}\big[(H-z_0 I_{\cH})\wti f_m-f_{m-1} \big] = 0.
\end{equation}
This in turn implies
\begin{equation}\label{e4.36}
(H-z_0 I_{\cH})(H_0-z_0 I_{\cH})^{-1}\big[(H-z_0 I_{\cH})\wti f_m-f_{m-1} \big] = (H-z_0 I_{\cH})\wti f_m-f_{m-1}.
\end{equation}
Using \eqref{fm} and \eqref{e4.36} one obtains 
\begin{align} 
& (H-z_0 I_{\cH})f_m = (H-z_0 I_{\cH})\wti f_m   \no \\
& \qquad -(H-z_0 I_{\cH})(H_0-z_0 I_{\cH})^{-1}\big[(H-z_0 I_{\cH})\wti f_m-f_{m-1} \big]   \no \\
& \quad =f_{m-1}. 
\end{align} 
Moreover, with \eqref{fm} and \eqref{opop} one deduces 
$V_1 f_m = V_1\wti f_m = \varphi_m$.
\end{proof}
%%%%%%

For $k=1$, Theorem~\ref{t4.5} reduces to the classical Birman--Schwinger principle.
In addition, it is shown in the next corollary that a resolvent-type formula as in \cite[Lemma B.1]{Fr18} 
or \cite[eq.~(2.13)]{GLMZ05} holds under our weak assumption Hypothesis~\ref{h4.1} for all those 
$z_0\in\rho(H_0)$ with $z_0\not\in\sigma_p(H)$ and $\psi\in\ran(H-z_0 I_{\cH})$.

%%%%%%
\begin{corollary}\lb{c4.6}
Let $H_0$ and $H=H_0+V_2^*V_1$, $\dom(H) = \dom(H_0)$, be as in Hypothesis~\ref{h4.1} and let 
$z_0\in\rho(H_0)$. Then one has the following. \\[1mm] 
$(i)$ If $z_0\in\sigma_p(H)$ and $f_0$ is an eigenvector of $H$ then 
$0\in\sigma_p\big( I_{\cK} + V_1(H_0-z_0 I_{\cH})^{-1}V_2^*\big)$ and 
$V_1 f_0$ is an eigenvector of $\big[ I_{\cK} + V_1(H_0-z_0 I_{\cH})^{-1}V_2^*\big]$ with 
eigenvalue zero. \\[1mm] 
$(ii)$ If $0\in\sigma_p\big( I_{\cK} + V_1(H_0-z_0 I_{\cH})^{-1}V_2^*\big)$ and 
$\varphi_0$ is a corresponding eigenvector, then $z_0\in\sigma_p(H)$ and $f_0=-(H_0-z_0 I_{\cH})^{-1}V_2^*\varphi_0$ is an eigenvector of $H$. \\[1mm] 
$(iii)$ The geometric multiplicity of $z_0\in\sigma_p(H)$ and $0\in\sigma_p\big( I_{\cK} + V_1(H_0-z_0 I_{\cH})^{-1}V_2^*\big)$ coincide, that is,
\begin{equation}
m_g(z_0; H) = m_g\big(0; I_{\cK} + V_1(H_0-z_0 I_{\cH})^{-1}V_2^*\big). 
\lb{4.46} 
\end{equation}
$(iv)$ If $z_0\not\in\sigma_p(H)$ and $\psi \in \cH$, then
\begin{align}
\begin{split} 
& \psi\in\ran (H-z_0 I_{\cH}) \, \text{ if and only if } \\
& \quad V_1(H_0-z_0 I_{\cH})^{-1}\psi\in \ran \big( I_{\cK} + V_1(H_0-z_0 I_{\cH})^{-1}V_2^*\big)  \
\end{split} 
\end{align}
and for all $\psi\in\ran (H-z_0 I_{\cH})$ one has
\begin{align}
\begin{split} 
& (H-z_0 I_{\cH})^{-1}\psi=(H_0-z_0 I_{\cH})^{-1}\psi    \\
& \quad -(H_0-z_0 I_{\cH})^{-1}V_2^*\big[ I_{\cK} + V_1(H_0-z_0 I_{\cH})^{-1}V_2^*\big]^{-1} V_1(H_0-z_0 I_{\cH})^{-1}\psi.
\end{split} 
\end{align}
\end{corollary}
%%%%%%
\begin{proof}
Items $(i)$ and $(ii)$ follow immediately from Theorem~\ref{t4.5} when considering the special case $k=1$.
It follows from the proof of Theorem~\ref{t4.5} that $V_1$ maps 
$\ker(H_0-z_0 I_{\cH})$ bijectively onto 
$\ker(I_{\cK} + V_1(H_0-z_0 I_{\cH})^{-1}V_2^*)$.
This implies $(iii)$.
For the statements in $(iv)$ we first assume that $\psi\in\ran (H-z_0 I_{\cH})$, that is, for some $\varphi\in\dom (H)$ one has
\begin{equation}
 \psi=(H-z_0 I_{\cH})\varphi=V_2^*V_1\varphi+(H_0-z_0 I_{\cH})\varphi.
\end{equation}
Then
\begin{align}
\begin{split} 
V_1(H_0-z_0 I_{\cH})^{-1}\psi &= V_1(H_0-z_0 I_{\cH})^{-1}\big(V_2^*V_1+(H_0-z_0 I_{\cH})\big)\varphi   \\
&=\big[I_{\cK} + V_1(H_0-z_0 I_{\cH})^{-1}V_2^*\big] V_1\varphi
\end{split} 
\end{align}
shows that $V_1(H_0-z_0 I_{\cH})^{-1}\psi\in\ran \big( I_{\cK} + V_1(H_0-z_0 I_{\cH})^{-1}V_2^*\big)$. Conversely, if $z_0\not\in\sigma_p(H)$ then by $(ii)$ the operator 
$\big[ I_{\cK} + V_1(H_0-z_0 I_{\cH})^{-1}V_2^*\big]$ is invertible and if
$\psi \in \cH$ with 
$V_1(H_0-z_0 I_{\cH})^{-1}\psi\in \ran \big( I_{\cK} + V_1(H_0-z_0 I_{\cH})^{-1}V_2^*\big)$ then the vector
\begin{align}
\begin{split} 
 \xi=(H_0-z_0 I_{\cH})^{-1}\psi & - (H_0-z_0 I_{\cH})^{-1}V_2^* 
 \big[I_{\cK} +V_1(H_0-z_0 I_{\cH})^{-1}V_2^*\big]^{-1}   \\ 
& \quad \times  V_1(H_0-z_0 I_{\cH})^{-1}\psi
\end{split}
\end{align}
is well-defined. A straightforward computation using $H-z_0 I_{\cH} = H_0-z_0 I_{\cH}+V_2^*V_1$ shows that 
$(H-z_0 I_{\cH})\xi=\psi$.
Hence $\psi\in\ran(H-z_0 I_{\cH})$. Moreover, 
as $z_0\not\in\sigma_p(H)$ this implies $\xi=(H-z_0 I_{\cH})^{-1}\psi$ and the last assertion follows from 
the definition of $\xi$.
\end{proof}
%%%%%%

%%%%%%
\begin{remark} \lb{r4.7} 
Let $z_0\not\in\sigma_p(H)$ and assume that $[ I_{\cK} + V_1(H_0-z_0 I_{\cH})^{-1}V_2^*]^{-1} \in \cB(\cK)$. 
Then by Corollary~\ref{c4.6} one has $\ran(H-z_0 I_{\cH})=\cH$ and the formula
 \begin{align}
 \begin{split} 
& (H-z_0 I_{\cH})^{-1}=(H_0-z_0 I_{\cH})^{-1}   \\
& \quad -(H_0-z_0 I_{\cH})^{-1}V_2^* 
 \big[ I_{\cK} + V_1(H_0-z_0 I_{\cH})^{-1}V_2^*\big]^{-1} V_1(H_0-z_0 I_{\cH})^{-1}   \lb{3.22}
 \end{split}
\end{align}
holds on $\cH$. Note that $V_1(H_0-z_0 I_{\cH})^{-1}\in\cB(\cH,\cK)$ since $V_1$ is closed by Hypothesis~\ref{h4.1}. 
Similarly $V_2^* \big[ I_{\cK} + V_1(H_0-z_0 I_{\cH})^{-1}V_2^*\big]^{-1} \in \cB(\cK,\cH)$ since $V_2^*$ is closed.
Therefore
the right-hand side of \eqref{3.22} is a bounded and everywhere defined operator in 
$\cH$ and hence also $(H-z_0 I_{\cH})^{-1}\in\cB(\cH)$. This implies, in particular, that $H$ is closed and $z_0\in\rho(H)$. \hfill $\diamond$ 
\end{remark}
%%%%%%

Assume again that $H_0$ and $H=H_0+V_2^*V_1$ are as in Hypothesis~\ref{h4.1} and suppose that $\rho(H)\cap\rho(H_0)\not=\emptyset$.
Since $V_1$ is closed one has $V_1(H_0-zI_\cH)^{-1}\in\cB(\cH,\cK)$ for $z\in\rho(H_0)$ and $V_1(H-zI_\cH)^{-1}\in\cB(\cH,\cK)$ for $z\in\rho(H)$
and it is easy to verify that the resolvent formulas
\begin{equation}\lb{hh1}
 (H-zI_\cH)^{-1}=(H_0-zI_\cH)^{-1}-(H_0-zI_\cH)^{-1}V_2^*V_1(H-zI_\cH)^{-1}
\end{equation}
and
\begin{equation}\lb{hh2}
 (H-zI_\cH)^{-1}=(H_0-zI_\cH)^{-1}-(H-zI_\cH)^{-1}V_2^*V_1(H_0-zI_\cH)^{-1}
\end{equation}
hold for all $z\in\rho(H_0)\cap\rho(H)$.
Multiplying \eqref{hh1} from the left by $V_1$ and from the right by $V_2^*$ leads to
\begin{equation}
 \bigl[I_{\cK}+V_1(H_0-zI_\cH)^{-1}V_2^*\bigr] \bigl[I_{\cK}-V_1(H-zI_\cH)^{-1}V_2^*\bigr]=I_{\cK}
\end{equation}
and multiplying \eqref{hh2} from the left by $V_1$ and from the right by $V_2^*$ leads to
\begin{equation}
 \bigl[I_{\cK}-V_1(H-zI_\cH)^{-1}V_2^*\bigr] \bigl[I_{\cK}+V_1(H_0-zI_\cH)^{-1}V_2^*\bigr]=I_{\cK}
\end{equation}
for all $z\in\rho(H_0)\cap\rho(H)$. Note that the above identities hold on $\dom(V_2^*)$.  
In the next corollary we conclude under  an additional assumption that the individual factors are boundedly invertible.

%%%%%
\begin{corollary} \lb{c4.8}
Let $H_0$ and $H=H_0+V_2^*V_1$ be as in Hypothesis~\ref{h4.1}, suppose $\rho(H)\cap\rho(H_0)\not=\emptyset$ and assume, 
in addition, that
\begin{equation}\lb{assvv}
  \overline{V_1(H_0-zI_\cH)^{-1}V_2^*}\in\cB(\cK),\quad\overline{V_1(H-zI_\cH)^{-1}V_2^*}\in\cB(\cK), \quad z\in\rho(H)\cap\rho(H_0).
\end{equation}
Then $I_\cK+\overline{V_1(H_0-zI_\cH)^{-1}V_2^*}$ and $I_\cK-\overline{V_1(H-zI_\cH)^{-1}V_2^*}$ are boundedly invertible and one has 
\begin{equation}\lb{ia}
 \bigl[I_\cK+\overline{V_1(H_0-zI_\cH)^{-1}V_2^*}\bigr]^{-1}=I_\cK-\overline{V_1(H-zI_\cH)^{-1}V_2^*},\quad z\in\rho(H)\cap\rho(H_0),
\end{equation}
and
\begin{equation}\lb{ib}
 \bigl[I_\cK-\overline{V_1(H-zI_\cH)^{-1}V_2^*}\bigr]^{-1}=I_\cK+\overline{V_1(H_0-zI_\cH)^{-1}V_2^*},\quad z\in\rho(H)\cap\rho(H_0).
\end{equation}
The assumption \eqref{assvv} holds, in particular, when $V_2\in\cB(\cH,\cK)$. In this case one has $V_1(H_0-zI_\cH)^{-1}V_2^*\in\cB(\cK)$
and  $V_1(H-zI_\cH)^{-1}V_2^*\in\cB(\cK)$ for $z\in\rho(H)\cap\rho(H_0)$ and the closures in the identities \eqref{ia} and \eqref{ib} can be omitted.
\end{corollary}
%%%%%%

%%%%%%
\begin{remark} \lb{r4.9} 
$(i)$ We emphasize that Hypothesis~\ref{h4.1}, in particular, permits the case where $\cH$ and $\cK$ differ. But it also includes the situation 
$\cK = \cH$, $V_2 = I_{\cH}$, and $V = V_1$, which then naturally leads to the non-symmetrized Birman--Schwinger operator family 
$I_{\cH} + V(H_0 - z I_{\cH})^{-1}$, $z \in \rho(H_0)$, in $\cH$. 
\\[1mm]
$(ii)$
A slightly more general situation than in $(i)$ consists of $\cK$ being a closed subspace of
$\cH$, $V_2=P_{\cK}$ the orthogonal projection onto $\cK$ and hence $V_2^*=\iota_{\cK}$ 
the canonical embedding of $\cK$ in $\cH$, and $V=V_1$ a closed operator from $\cH$ into $\cK$, leading to 
the Birman--Schwinger operator family 
\begin{equation}\lb{bssmall}
I_{\cK} + V(H_0 - z I_{\cH})^{-1}\iota_\cK, \quad z \in \rho(H_0),
\end{equation}
in $\cK\subset\cH$.
For the characterization of points in $\sigma_p(H)\cap\rho(H_0)$ as in Corollary~\ref{c4.6}~(i) one may also use the corresponding extended
Birman--Schwinger operator family in $\cH= \cK \oplus \cK^{\bot}$ given by
\begin{equation}\lb{bsbig}
I_{\cH} + \big[V(H_0 - z I_{\cH})^{-1}\iota_\cK \oplus 0_{\cK^{\bot}}\big], \quad z \in \rho(H_0),
\end{equation}
since $0$ is an eigenvalue of the operator in \eqref{bssmall} if and only if $0$ is an eigenvalue of the operator in \eqref{bsbig} (cf.\ Example~\ref{e4.11}). \\[1mm]
$(iii)$ The Birman--Schwinger principle has been derived in Corollary~\ref{c4.6}\,$(i)$,~$(ii)$ under very general conditions on $H_0$ and $H$, in particular, the latter was not assumed to be closed. However, one can also consider very general situations in a rather different direction going far beyond relatively bounded perturbations $V$ with respect to $H_0$ so that the domain of $H$ can no longer be compared to that of $H_0$. In fact, one can even consider situations more general than quadratic form perturbations, following the lead in Kato's paper \cite{Ka66}. Indeed, Konno and Kuroda \cite{KK66} considered this general setup in the case where $H_0$ is self-adjoint, and \cite{BGHN17},  \cite{GHN15}, \cite{GLMZ05} in the situation where $H_0$ and $H$ are non-self-adjoint.  In particular, formula \eqref{4.46} extends to such a generalized setup. \hfill $\diamond$
\end{remark}
%%%%%%

We conclude with a brief discussion of multi-dimensional Schr\"odinger operators.

%%%%%%
\begin{example} \lb{e4.11} 
 Consider the selfadjoint operator $H_0=-\Delta$, $\dom (H_0) = H^2(\bbR^n)$, in $\cH=L^2(\bbR^n; d^nx)$ and let $v:\bbR^n\rightarrow\bbC$ be a $($Lebesgue\,$)$ measurable function such that 
 \begin{equation}\lb{h2ass}
 vh\in L^2(\bbR^n; d^nx) \, \text{ for all } \, h\in H^2(\bbR^n).
 \end{equation}
 Next define the functions $v_1$ and $v_2$ by
 \begin{equation}
  v_1= \sgn (v) |v|^{1/2}  \, \text{ and } \, v_2= |v|^{1/2} , \quad 
  \sgn (v(x)):=\begin{cases} v(x)/|v(x)|, & \text{if } v(x)\not=0,\\ 0, & \text{if } v(x)=0.\end{cases}
 \end{equation}
The maximal multiplication operators associated to $v$ and $v_1$ in $L^2(\bbR^n; d^nx)$ are denoted by $V$ and $V_1$, respectively. Furthermore, let
\begin{equation}
 \cK_0=\big\{v_2 f \, \big| \, f\in L^2(\bbR^n; d^nx) v_2 f\in L^2(\bbR^n; d^nx) \big\}, 
\end{equation}
and 
\begin{equation}
 \cK:=\overline\cK_0\subseteq  L^2(\bbR^n; d^nx).
\end{equation}
The maximal multiplication operator associated to $v_2$ is viewed as a densely defined operator from
$L^2(\bbR^n; d^nx)$ to $\cK$ and will be denoted by $V_2$; the adjoint from $\cK$ to $L^2(\bbR^n;d^nx)$ is the maximal multiplication operator $V_2^* = v_2$. One observes that $\ran (V_1) \subseteq \cK$ and hence also $V_1$ can 
be viewed as a maximal multiplication operator
from $L^2(\bbR^n; d^nx)$ to $\cK$. In this situation one has $V=V_2^*V_1$ and 
$\dom (H_0) \subseteq \dom (V_2^*V_1)$ by \eqref{h2ass}. Therefore, Hypothesis~\ref{h4.1}
is satisfied and Theorem~\ref{t4.5} can be applied to give a description of the Jordan chains of the Schr\"{o}dinger operator in $L^2(\bbR^n; d^nx)$
\begin{equation}\lb{schr}
 H=H_0+V,\quad \dom (H) = H^2(\bbR^n),
\end{equation}
in terms of the Birman--Schwinger operator in $\cK$, 
\begin{equation}
I_\cK+\sgn (v) |v|^{1/2}  \big(-\Delta-z I_{L^2(\bbR^n; d^nx)}\big)^{-1}|v|^{1/2}.    \lb{4.62} 
\end{equation}
In fact, writing $L^2(\bbR^n; d^nx) = \cK \oplus \cK^{\bot}$, the operator in \eqref{4.62} naturally extends to the 
classical Birman--Schwinger operator in $L^2(\bbR^n; d^nx)$ of the type 
\begin{equation}
 I_{\cK \oplus \cK^{\bot}}+\sgn (v) |v|^{1/2}  \big(-\Delta-z I_{L^2(\bbR^n; d^nx)}\big)^{-1}|v|^{1/2} \oplus 0_{\cK^{\bot}}    \lb{4.63} 
\end{equation}
$($with $0_{\cK^{\bot}}$ the zero operator in $\cK^{\bot}$$)$, which is typically just denoted by  
\begin{equation}
I_{L^2(\bbR^n; d^nx)}+ \sgn (v) |v|^{1/2}  \big(-\Delta-z I_{L^2(\bbR^n; d^nx)}\big)^{-1}|v|^{1/2};    \lb{4.64} 
\end{equation}
see Remark~\ref{r4.9}\,$(ii)$.
Under the present weak assumptions on the potential
$V$, the Schr\"{o}dinger operator $H$ in \eqref{schr} may not be closed and its spectrum may cover the whole complex plane. Of course, under additional assumption on $V$ the situation is markedly different. For instance, assume that $n=1,2,3$ and $V\in L^2(\bbR^n; d^nx)$, or $n\geq 4$ and $V\in L^p(\bbR^n)$ with $p>n/2$. Then by \cite[Theorem 11.2.11]{Da07}, $H$ is a relatively compact perturbation of $H_0$ and it follows that
the spectrum of $H$ in $\bbC \backslash [0,\infty)$ consists of discrete eigenvalues with finite algebraic multiplicity. The algebraic eigenspaces can then be characterized with the help of Theorem~\ref{t4.5}.
\end{example}
%%%%%%

%%%%%%
%%%%%%
\section{The Index of Meromorphic Operator-Valued Functions and Applications to Algebraic Multiplicities of Analytic Families} \lb{s5}
%%%%%%
%%%%%%

In the following we briefly recall the notion of the index of meromorphic operator-valued functions and discuss applications to the algebraic multiplicity of a zero of finite-type of an analytic operator-valued function following 
\cite{BGHN16}, \cite{BGHN17}, \cite{GHN15} \cite{GS71} \cite[Sects.\ XI.8, XI.9]{GGK90}, \cite[Ch.\ 4]{GL09}.  

We begin with some preparatory material taken from \cite{BGHN16}. 
Let $\cH$ be a separable complex Hilbert space,
assume that $\Omega \subseteq \mathbb C$ is an open set, and let $M(\dott)$ be a $\cB(\cH)$-valued 
meromorphic function on $\Omega$ that has the norm convergent Laurent expansion around 
$z_0 \in \Omega$ of the form
\begin{equation}\lb{mmm}
 M(z) = \sum_{k= - N_0}^\infty (z - z_0)^{k} M_k(z_0),\quad z\in D(z_0;\varepsilon_0)\backslash\{z_0\}, 
\end{equation}
where $M_k(z_0) \in \cB(\cH)$, $k\in\mathbb Z$, $k\geq -N_0$ and $\varepsilon_0>0$ is sufficiently small such that
the punctured open disc $D(z_0;\varepsilon_0)\backslash\{z_0\}$
is contained in $\Omega$. The {\it principal $($or singular\,$)$ part} ${\rm pp}_{z_0} \{M(z)\}$ of $M(\dott)$ at $z_0$ is defined as the finite sum 
\begin{equation}\lb{koko}
 {\rm pp}_{z_0} \{M(z)\}=\sum_{k= - N_0}^{-1} (z - z_0)^{k} M_k(z_0).
\end{equation}

%%%%%%
\begin{definition} \lb{d5.1}
Let $\Omega \subseteq\bbC$ be an open set and let $M(\dott)$ be a $\cB(\cH)$-valued 
meromorphic function on $\Omega$. Then  
$M(\dott)$ is called {\it finitely meromorphic at $z_0 \in \Omega$} if $M(\dott)$ is 
analytic on the punctured disk $D(z_0;\varepsilon_0)\backslash\{z_0\} \subset \Omega$ 
with sufficiently small $\varepsilon_0 > 0$, and the principal part ${\rm pp}_{z_0} \{M(z)\}$ of $M(\dott)$ at $z_0$ is of finite rank, that is, 
the principal part of $M(\dott)$ is of the type \eqref{koko}, and one has 
\begin{equation}\lb{lkj}
M_k(z_0) \in \cF(\cH), \quad -N_0 \leq k \leq -1. 
\end{equation}
The function $M(\dott)$ is called {\it finitely meromorphic on $\Omega$} if it is meromorphic 
on $\Omega$ and finitely meromorphic at each of its poles.
\end{definition}
%%%%%%

In the context of Theorem~\ref{t3.10}, that is, 
$M \colon \Omega \to \cB(\cH)$ is analytic on $\Omega$ and 
$z_0 \in \Omega$ is a zero of finite-type of $M(\dott)$,  
an application of the analytic Fredholm Theorem~\ref{tA.2} 
(cf.\, e.g., \cite[Sect.~4.1]{GL09},  \cite{GS71}, \cite{Ho70}, \cite[Theorem\ VI.14]{RS80}, \cite{St68}) 
shows that $M(\dott)^{-1}$ is finitely meromorphic at $z_0$.  

Assume that $M_j(\dott)$, $j=1,2$, are $\cB(\cH)$-valued 
meromorphic functions on $\Omega$ that are both finitely meromorphic at $z_0 \in \Omega$,
choose $\varepsilon_0 > 0$ such that \eqref{mmm} holds for both functions $M_j(\dott)$, and let $0<\varepsilon<\varepsilon_0$.
Then by \cite[Lemma\ XI.9.3]{GGK90} or \cite[Proposition\ 4.2.2]{GL09} also the functions 
$M_1(\dott) M_2(\dott)$ and $M_2(\dott) M_1(\dott)$ are finitely meromorphic at 
$z_0 \in \Omega$, the operators 
\begin{equation}\lb{ws}
\ointctrclockwise_{\partial D(z_0; \varepsilon)} d \zeta \, M_1(\zeta) M_2(\zeta)  \quad\text{and}\quad \ointctrclockwise_{\partial D(z_0; \varepsilon)} d \zeta \, M_2(\zeta) M_1(\zeta)
\end{equation}
are both of finite rank and the identities
\begin{align}
& {\tr}_{\cH} \bigg(\ointctrclockwise_{\partial D(z_0; \varepsilon)} d \zeta \, 
M_1(\zeta) M_2(\zeta)\bigg) 
= {\tr}_{\cH} \bigg(\ointctrclockwise_{\partial D(z_0; \varepsilon)} d \zeta \, 
M_2(\zeta) M_1(\zeta)\bigg),   \lb{wss} \\
& {\tr}_{\cH} \big({\rm pp}_{z_0} \, \{M_1(z) M_2(z)\}\big) = 
{\tr}_{\cH} \big({\rm pp}_{z_0} \, \{M_2(z) M_1(z)\}\big), \quad 
0 < |z-z_0| < \varepsilon_0,     \lb{4.7}
\end{align}
hold; here the symbol $ \ointctrclockwise$ denotes the contour integral and
$\partial D(z_0; \varepsilon)$ is the counterclockwise oriented circle 
with radius $\varepsilon$ centered at $z_0$. 

Next, we take a closer look at resolvents of closed operators providing prime examples of finitely meromorphic 
operator-valued functions, but first we introduce the extended resolvent set $\wti \rho(A)$ of $A$ as follows:

%%%%%%%
\begin{definition} \lb{d5.2}
Let $A$ be a closed operator in $\cK$. Then the {\it extended resolvent set of $A$} is defined by
\begin{align}
\begin{split} 
\wti \rho(A) = \rho(A) \cup \sigma_d(A).     \lb{5.7} 
\end{split} 
\end{align}
\end{definition}
%%%%%%%

One verifies that 
\begin{equation}
\rho(A), \; \wti \rho(A) \, \text{ are open subsets of $\bbC$.}    \lb{5.8}
\end{equation}
Since $ \wti \rho(A) \subset \bbC$ is open, it's connected components are open and at most countable 
(cf., e.g., \cite[Theorem~2.9]{Co78}).

Next, suppose $A$ is a closed operator in $\cK$. 
If $\lambda_0\in \wti \rho(A)$, then, according to \cite[Sect.~III.6.5]{Ka80}, the singularity structure of
the resolvent of $A$ near $\lambda_0$ is of the very special norm convergent meromorphic type, 
\begin{align}
\begin{split} 
(A - zI_{\cK})^{-1}&=(\lambda_0-z)^{-1}P(\lambda_0;A)
+\sum_{k=1}^{\mu(\lambda_0;A)} (\lambda_0-z)^{-k-1}
(-1)^k F(\lambda_0;A)^k   \\
& \quad +\sum_{k=0}^\infty (\lambda_0-z)^k (-1)^k
S(\lambda_0;A)^{k+1} \lb{6.7}
\end{split} 
\end{align}
for $z \in \bbC$ in a sufficiently small punctured neighborhood of $\lambda_0$. Here
\begin{align}
F(\lambda_0;A)&=(A - \lambda_0 I_{\cK})P(\lambda_0;A)=\f{1}{2\pi i}
\ointctrclockwise_{\partial D(\lambda_0; \varepsilon) } d\zeta \,
(\lambda_0-\zeta)(A - \zeta I_{\cK})^{-1} \in \cB(\cK),
\lb{6.8} \\
S(\lambda_0;A)&=-\f{1}{2\pi i} \ointctrclockwise_{\partial D(\lambda_0; \varepsilon) }
d\zeta \, (\lambda_0-\zeta)^{-1}(A - \zeta I_{\cK})^{-1} \in
\cB(\cK), \lb{6.9}
\end{align}
and $F(\lambda_0;A)$ is nilpotent with its range contained in that of $P(\lambda_0;A)$,
\begin{equation}
F(\lambda_0;A)=P(\lambda_0;A) F(\lambda_0;A) = F(\lambda_0;A)P(\lambda_0;A). \lb{6.10}
\end{equation}
Of course, $P(\lambda_0;A) = F(\lambda_0;A) = 0$ if $\lambda_0 \in \rho(A)$. Moreover,
\begin{align}
\begin{split}
&S(\lambda_0;A) A \subseteq AS(\lambda_0;A), \quad
(A - \lambda_0 I_{\cK}) S(\lambda_0;A)
= I_{\cK}-P(\lambda_0;A), \\
&S(\lambda_0;A)P(\lambda_0;A)=
P(\lambda_0;A)S(\lambda_0;A)=0. \lb{6.11}
\end{split}
\end{align}
Finally,
\begin{align}
&\mu(\lambda_0;A) + 1 \leq m_a(\lambda_0;A)=\dim(\ran(P(\lambda_0;A))),    \lb{6.12} \\
&m_a(\lambda_0;A) = {\tr}_{\cK}(P(\lambda_0;A)), \quad F(\lambda_0;A)^{m_a(\lambda_0;A)} = 0. \lb{6.13}
\end{align}

In particular, since $F(\lambda_0;A)$ (and hence, $F(\lambda_0;A)^k$, $k \in \bbN$) is of finite rank and nilpotent, \eqref{6.7}, \eqref{6.12}, and \eqref{6.13} prove the following fact. 

%%%%%%
\begin{lemma} \lb{l6.2}
Assume that $A$ is closed in $\cK$. \\[1mm]
$(i)$ Then the map 
$z \mapsto (A - z I_{\cK})^{-1}$ is analytic on $\rho(A)$ and finitely meromorphic on $\wti \rho(A)$. \\[1mm]
$(ii)$ 
If $\cH$ is a Hilbert space, $S \in\cB(\cK,\cH)$, $T \in\cB(\cH,\cK)$, and $z_0 \in \sigma_d(A)$, then the 
$\cB(\cH)$-valued function 
 \begin{equation}\lb{sts}
  z\mapsto S(A - zI_{\cK})^{-1} T,  \quad z\in \wti \rho(A),
 \end{equation}
is finitely meromorphic at $z_0$. 
\end{lemma}
%%%%%%

To introduce the notion of an index of $M(\dott)$ we need the following set of assumptions:

%%%%%%%%%
\begin{hypothesis} \lb{h5.2} 
Let $\Omega \subseteq \bbC$ be open and connected, and $\Omega_d \subset \Omega$ a discrete set 
$($i.e., a set without limit points in $\Omega$$)$. Suppose that $M:\Omega \backslash \Omega_d \to \cB(\cH)$ 
is analytic and that $M(\dott)$ is finitely meromorphic on $\Omega$. In addition, suppose that 
\begin{equation}
M(z) \in \Phi(\cH) \, \text{ for all } \, z\in\Omega \backslash \Omega_d,    \lb{4.1} 
\end{equation}  
and for all $z_0 \in \Omega_d$ there is a norm convergent Laurent expansion around $z_0$
of the form
\begin{equation}
M(z) = \sum_{k=-N_0}^{\infty} (z-z_0)^k M_k(z_0), \quad 
0 < |z - z_0| < \varepsilon_0,    \lb{4.2}
\end{equation}
for some $N_0 = N_0(z_0) \in \bbN$ and some $0 < \varepsilon_0 = \varepsilon_0(z_0)$ 
sufficiently small, with  
\begin{align}\lb{4.2a} 
\begin{split} 
& M_{-k}(z_0) \in \cF(\cH), \; 1 \leq k \leq N_0(z_0),  \quad  M_0(z_0) \in \Phi(\cH), \\
& M_k(z_0) \in \cB(\cH), \; k \in \bbN. 
\end{split} 
\end{align} 
Finally, given $z_0\in \Omega_d$, assume that $M(\dott)$ is boundedly invertible on 
$D(z_0; \varepsilon_0) \backslash \{z_0\}$ for some $0 < \varepsilon_0$ sufficiently small. 
\end{hypothesis}
%%%%%%%%% 

Recalling the meromorphic Fredholm Theorem~\ref{tA.4}, enables one to make the following definition of 
the index of $M(\dott)$:

%%%%%%%%%
\begin{definition} \lb{d5.3} 
Assume Hypothesis~\ref{h5.2},  let $z_0 \in \Omega$, and suppose that $0 < \varepsilon$ is sufficiently small. 
Then the {\it index of $M(\dott)$ with respect to the counterclockwise 
oriented circle $\partial D(z_0; \varepsilon)$}, $\ind_{\partial D(z_0; \varepsilon)}(M(\cdot))$, is defined by 
 \begin{align}
\begin{split}
\ind_{\partial D(z_0; \varepsilon)}(M(\cdot)) &= {\tr}_{\cH}\bigg(\f{1}{2\pi i} 
\ointctrclockwise_{\partial D(z_0; \varepsilon)} d\zeta \, 
M'(\zeta) M(\zeta)^{-1}\bigg)     \lb{4.8} \\
& = {\tr}_{\cH}\bigg(\f{1}{2\pi i} 
\ointctrclockwise_{\partial D(z_0; \varepsilon)} d\zeta \, 
M(\zeta)^{-1} M'(\zeta)\bigg), \quad 0 < \varepsilon < \varepsilon_0. 
\end{split}
\end{align}  
\end{definition}
%%%%%%%%%

A special case of the logarithmic residue theorem proved by Gohberg and Sigal \cite[Theorems~2.1, 2.1']{GS71} (see also \cite[Sect.~XI.9]{GGK90} and \cite[Sect.~4.4]{GL09}) then reads as follows.

%%%%%%
\begin{theorem} \lb{t5.4}
Let $\Omega \subseteq \bbC$ be open and $z_0 \in \Omega$. Suppose that $M(\dott) : \Omega \to \cB(\cH)$ 
satisfies Hypothesis~\ref{h5.2}. Then 
\begin{equation}
{\ind}_{\partial D(z_0; \varepsilon)}(M(\dott)) \in \bbZ.
\end{equation}
In addition, assume that $M(\dott) : \Omega \to \cB(\cH)$ is analytic at $z_0$. Then
\begin{equation}
m_a(z_0; M(\dott)) = {\ind}_{\partial D(z_0; \varepsilon)}(M(\dott)) \in \bbN_0.    \lb{4.20}
\end{equation}
\end{theorem}
%%%%%%

While the algebraic multiplicity computations in Examples~\ref{e3.4}, \ref{e3.6}--\ref{e3.8}, employing Jordan chains of 
maximal length, is somewhat cumbersome, we now revisit these computations using equality \eqref{4.20} and show that the 
latter yields algebraic multiplicities in an effortless manner.

%%%%%%
\begin{example}[Examples~\ref{e3.4}, \ref{e3.6}--\ref{e3.8} revisited] \lb{e5.5}
Noting that the the matrix functions $A_1(\dott)$, $A_2(\dott)$, $A_k(\dott)$, and $A_{k_1,\dots,k_N}(\dott)$ all satisfy Hypothesis~\ref{h5.2} with $\Omega=\bbC$ and $\Omega_d=\{0\}$, one trivially computes 
\begin{align}
& A_1'(\zeta) A_1(\zeta)^{-1} = \diag\big(\zeta^{-1},0,0\big), \quad \zeta \in \bbC \backslash \{0\},  \\
&  A_2'(\zeta) A_2(\zeta)^{-1} = \diag \big(\zeta^{-1},2\, \zeta^{-1},0\big), \quad \zeta \in \bbC \backslash \{0\},  \\
&  A_k'(\zeta) A_k(\zeta)^{-1} = \diag \big(\zeta^{-1},k\, \zeta^{-1},0\big), \quad \zeta \in \bbC \backslash \{0\},  \\
& A_{k_1,\dots,k_N}'(\zeta) A_{k_1,\dots,k_N}(\zeta)^{-1}=\diag\big(k_1\,\zeta^{-1},\dots,k_N\,\zeta^{-1},0\big), \quad \zeta \in \bbC \backslash \{0\}.   
\end{align}
Hence
\begin{align}
& \tr_{\bbC^3} \big(A_1'(\zeta) A_1(\zeta)^{-1}\big) = \zeta^{-1}, \quad \zeta \in \bbC \backslash \{0\},  \\
& \tr_{\bbC^3} \big(A_2'(\zeta) A_2(\zeta)^{-1}\big) = 3 \, \zeta^{-1}, \quad \zeta \in \bbC \backslash \{0\},  \\
& \tr_{\bbC^3} \big(A_k'(\zeta) A_k(\zeta)^{-1}\big) = (k+1) \zeta^{-1}, \quad \zeta \in \bbC \backslash \{0\},  \\
& \tr_{\bbC^{N+1}} \big(A_{k_1,\dots,k_N}'(\zeta) A_{k_1,\dots,k_N}(\zeta)^{-1}\big) 
= \bigg(\sum_{j=1}^N k_j\bigg) \zeta^{-1}, \quad \zeta \in \bbC \backslash \{0\},  
\end{align}
which together with the elementary fact $\f{1}{2\pi i} 
\ointctrclockwise_{\partial D(0; \varepsilon)} d\zeta\, \zeta^{-1} =1$ instantly yields the algebraic multiplicities 
$m_a(0; A(\dott))$ recorded in \eqref{3.21}, \eqref{3.29}, \eqref{3.32}, and \eqref{3.39}. 

For completeness we mention that the matrix function $A_\infty(\dott)$ in Example~\ref{e3.5} does not satisfy Hypothesis~\ref{h5.2} since $A_\infty(\dott)$ is nowhere invertible on $\bbC$. 
\end{example}
%%%%%%

In the remainder of this section we apply this circle of ideas to the pair of operators $(H,H_0)$ as discussed 
in Section~\ref{s4} and to the associated Birman--Schwinger operator $K(z) = - V (H_0 - z I_{\cH})^{-1}$, 
$z \in \rho(H_0)$. The following hypothesis is convenient and natural in the present context.

%%%%%%
\begin{hypothesis} \lb{h5.6}
Suppose $H_0$ is a closed operator in $\cH$ with $\rho(H_0) \neq \emptyset$, and assume 
that $V$ is an operator in $\cH$ satisfying $\dom(V) \supseteq \dom(H_0)$. Introduce
\begin{equation}
H = H_0 + V, \quad \dom(H) = \dom(H_0),
\end{equation}
and suppose that $\rho(H)\not=\emptyset$. Furthermore, it is assumed that the Birman--Schwinger operator is bounded, that is, 
\begin{equation}\lb{kbounded}
K(z) = - V (H_0 - z I_{\cH})^{-1} \in \cB(\cH) \, \text{ for some $($and hence for all\,$)$ $z\in\rho(H_0)$.}
\end{equation}
\end{hypothesis} 
%%%%%%

Assume Hypothesis~\ref{h5.6}. From 
$\rho(H)\not=\emptyset$ one concludes, in particular, that $H$ is a closed operator in $\cH$. 
We also note that assumption \eqref{kbounded} on the Birman--Schwinger operator is satisfied if $V$ is a closed operator in $\cH$, which is the typical situation in perturbation theory (see also \cite[Proposition~1]{GJ69} for 
related situations). In particular, in quantum mechanical 
applications to Schr\"odinger operators, $H_0$ is the self-adjoint realization of the Laplacian $- \Delta$ defined on $H^2(\bbR^n)$, $n \in \bbN$, and $V$ represents the maximally defined operator of multiplication with the function $V(\dott)$, which thus is closed in $L^2(\bbR^n; d^nx)$. Furthermore, if \eqref{kbounded} is satisfied for some $z\in\rho(H_0)$ then it is an direct consequence of the resolvent identity for $H_0$ that \eqref{kbounded} is 
satisfied for all $z\in\rho(H_0)$. One observes that in the notation of Section~\ref{s4} here $V=V_1$ and $V_2=V_2^*=I_\cH$, although we do not require 
here that $V$ is closed.

For $z\in\rho(H)\cap\rho(H_0)$ we have
the resolvent equation
\begin{equation}\lb{reso}
(H - zI_{\cH})^{-1} = (H_0 - zI_{\cH})^{-1} - (H - zI_{\cH})^{-1} V (H_0 - zI_{\cH})^{-1}.
\end{equation}
It follows from \eqref{kbounded} that 
\begin{equation}
I_{\cH} - K(z)=I_{\cH} + V (H_0 - z I_{\cH})^{-1},\quad  z\in\rho(H_0),
\end{equation}
and, as in the proof of Corollary~\ref{c4.8}, one verifies that for all $z\in\rho(H)\cap\rho(H_0)$ the operator $I_{\cH} - K(z)$
is boundedly invertible and one has
\begin{equation}\lb{4.29}
 (H_0 - z I_{\cH}) (H -z I_{\cH})^{-1}=I_\cH-V(H-zI_\cH)^{-1}=\big[I_{\cH} - K(z)\big]^{-1}\in \cB(\cH)
\end{equation}
and
\begin{equation}
(H -z I_{\cH})^{-1} = (H_0 - z I_{\cH})^{-1}\big[I_{\cH} - K(z)\big]^{-1}, \quad 
z \in \rho(H)\cap\rho(H_0).   \lb{4.29c}
\end{equation}
Inserting the latter expression for $(H -z I_{\cH})^{-1}$ in the right hand side of \eqref{reso} it follows that
\begin{align}
(H - z I_{\cH})^{-1} &= (H_0 - z I_{\cH})^{-1}   
- (H_0 - z I_{\cH})^{-1} \big[I_{\cH} - K(z)\big]^{-1}     
V (H_0 - z I_{\cH})^{-1},   \no \\
& \hspace*{6cm} z \in \rho(H)\cap\rho(H_0).    \lb{4.30}
\end{align} 

Theorem~\ref{t5.7} below provides a sufficient condition in terms of the Birman--Schwinger operator $K(\dott)$ 
for an isolated spectral point of $H$ to be a discrete eigenvalue. Moreover, the index formula \eqref{4.40} provides a convenient 
tool for computing algebraic multiplicities (cf.\ also the illustrations in the simple Example~\ref{e5.5}). 
In the special case discussed in 
Theorem~\ref{t5.9} later the algebraic multiplicity is also related to the algebraic 
multiplicity of the corresponding zero of the function $I_{\cH} - K(\dott)$. 
For our purposes it useful to recall from Lemma~\ref{l6.2} that the assumption $z_0\in\sigma_d(H_0)$ implies that the resolvent, 
\begin{equation}
\rho(H_0) \ni z\mapsto R_0(z) = (H_0-zI_\cH)^{-1} \in \cB(\cH), 
\end{equation} 
is finitely meromorphic on a disc $D(z_0; \varepsilon_0)$ for some $\varepsilon_0 > 0$ sufficiently small, and analytic on $D(z_0; \varepsilon_0)\backslash \{z_0\}$.
Using the Laurent expansion of the resolvent $R_0(\dott)$ it follows from the boundedness assumption \eqref{kbounded} in Hypothesis~\ref{h5.6}
that the function 
\begin{equation}
z\mapsto I_\cH-K(z)=I_\cH+V(H_0-zI_\cH)^{-1}
\end{equation}
is also finitely meromorphic on $D(z_0; \varepsilon_0)$; cf.\ Lemma~\ref{l6.2}~(ii). Hence the Birman--Schwinger operator $K(\dott)$ satisfies
\begin{equation}
K(z) = \sum_{k=-N_0}^{\infty} (z-z_0)^k K_k(z_0), \quad 
0 < |z - z_0| < \varepsilon_0,    \lb{4.32B}
\end{equation}
for some $N_0 = N_0(z_0) \in \bbN$  with  
\begin{equation}\lb{4.33B} 
K_{-k}(z_0) \in \cF(\cH), \; 1 \leq k \leq N_0(z_0), \quad 
K_k(z_0) \in \cB(\cH), \; k \in \bbN \cup \{0\}.
\end{equation} 
In the next theorem an additional assumption is imposed on the operator $K_0(z_0) \in \cB(\cH)$ in \eqref{4.32B}.

%%%%%%
\begin{theorem} \lb{t5.7}
Assume Hypothesis~\ref{h5.6} and consider a point $z_0 \in \sigma_d(H_0) \cap \sigma(H)$ such that
$D(z_0; \varepsilon_0)\backslash \{z_0\}\subset \rho(H_0)\cap\rho(H)$ for some $\varepsilon_0 > 0$. In addition, 
suppose that 
\begin{equation}
[I_{\cH} - K_0(z_0)] \in \Phi(\cH).    \lb{4.33C} 
\end{equation} 
Then $z_0$ is a discrete eigenvalue of $H$,  
\begin{equation}
z_0 \in \sigma_d(H),    \lb{4.39} 
\end{equation}
and for $0 < \varepsilon<\varepsilon_0$, 
\begin{equation}
m_a(z_0; H) = m_a(z_0; H_0) + {\ind}_{\partial D(z_0; \varepsilon)}(I_{\cH} - K(\dott)).   \lb{4.40} 
\end{equation} 
\end{theorem}
%%%%%%
\begin{proof}
We start with the proof of \eqref{4.39}: The assumptions $z_0 \in \sigma_d (H_0)$ and $D(z_0; \varepsilon_0)\backslash \{z_0\}\subset \rho(H_0)$ yield that 
both $R_0 (\dott)$ and $I_{\cH} - K(\dott)$ are analytic in $D(z_0;\varepsilon_0)\backslash \{z_0\}$ and finitely meromorphic on 
$D(z_0;\varepsilon_0)$ (cf.\ the discussion preceding Theorem~\ref{t5.7}). The assumption 
$D(z_0; \varepsilon_0)\backslash \{z_0\}\subset \rho(H)$ 
and \eqref{4.29} imply that $[I_{\cH} - K(z)]^{-1} \in \cB(\cH)$ and, in particular, $[I_{\cH} - K(z)] \in \Phi(\cH)$ for all $z\in D(z_0;\varepsilon_0)\backslash \{z_0\}$.
Due to the assumption \eqref{4.33C} it follows that \eqref{A.8} and hence Hypothesis~\ref{hA.3} is satisfied by $[I_{\cH} - K(\dott)]$
with the choice $\Omega = D(z_0;\varepsilon_0)$ and $\Omega_d=\{z_0\}$. It is clear that
alternative $(ii)$ in Theorem~\ref{tA.4} applies to $[I_{\cH} - K(\dott)]$ and hence
the function $[I_{\cH} - K(\dott)]^{-1}$ is finitely meromorphic on $D(z_0;\varepsilon_0)$.
From \eqref{4.30} and the fact that $R_0(\dott)$ and $VR_0(\dott)$ are finitely meromorphic on $D(z_0; \varepsilon_0)$ it follows that the resolvent of $H$, 
\begin{equation}
\rho(H) \ni z\mapsto R(z)=(H-zI_\cH)^{-1} \in \cB(\cH) 
\end{equation} 
is also finitely meromorphic on $D(z_0;\varepsilon_0)$ and analytic on $D(z_0;\varepsilon_0)\backslash\{z_0\}$ 
(since $D(z_0;\varepsilon_0)\backslash\{z_0\}\subset\rho(H)$ by assumption); one 
notes that $R(\dott)$ has a singularity at $z_0$ by the assumption $z_0\in\sigma(H)$. 
An application of \cite[Sect.\ III.6.5]{Ka80} then yields that $z_0 \in \sigma_p(H)$, and employing once more the finitely meromorphic property of 
$R(\cdot)$ on $D(z_0;\varepsilon_0)$, one infers that the Riesz projection associated with $z_0$, 
\begin{equation}
P(z_0;H)=\f{-1}{2\pi i} \ointctrclockwise_{\partial D(z_0;\varepsilon)} 
d \zeta \, (H - \zeta I_{\cH})^{-1},   \quad 0 < \varepsilon < \varepsilon_0,   \lb{4.41}
\end{equation}
is finite-dimensional. This in turn is equivalent to the eigenvalue $z_0$ having finite algebraic multiplicity,  
implying \eqref{4.39}.

Next, employing \eqref{4.29}, one obtains 
\begin{align}
& [I_{\cH} - K(z)]^{-1} [- K'(z)] = [I_{\cH} - K(z)]^{-1} V (H_0 - z I_{\cH})^{-2}   \no \\
& \quad =  [I_{\cH} - K(z)]^{-1} (-K(z))(H_0 - z I_{\cH})^{-1}   \no \\
& \quad = \big\{I_{\cH} - [I_{\cH} - K(z)]^{-1}\big\} (H_0 - z I_{\cH})^{-1}   \no \\
& \quad =  (H_0 - zI_{\cH})^{-1} - (H_0 - z I_{\cH}) (H - z I_{\cH})^{-1}  (H_0 - zI_{\cH})^{-1},     \lb{4.42} \\
& \hspace*{6.42cm}  z \in D(z_0, \varepsilon_0)\backslash \{z_0\}.   \no 
\end{align} 
Thus, for $0 < \varepsilon<\varepsilon_0$, \eqref{4.42} implies
\begin{align}
&  {\ind}_{\partial D(z_0; \varepsilon)}(I_{\cH} - K(\dott)) = \f{1}{2 \pi i} {\tr}_{\cH}\bigg( 
\ointctrclockwise_{\partial D(z_0; \varepsilon)} d\zeta \, 
[I_{\cH} - K(\zeta)]^{-1}[- K'(\zeta)]\bigg)     \no \\
& \quad = \f{1}{2 \pi i} {\tr}_{\cH}\bigg( 
\ointctrclockwise_{\partial D(z_0; \varepsilon)} d\zeta \, \big[(H_0 - \zeta I_{\cH})^{-1}     \no \\
& \hspace*{3.9cm} 
- (H_0 - \zeta I_{\cH}) (H - \zeta I_{\cH})^{-1} (H_0 - \zeta I_{\cH})^{-1}\big]\bigg)  \no \\ 
& \quad = - m_a(z_0; H_0)   \no \\
& \qquad - \f{1}{2 \pi i} {\tr}_{\cH}\bigg( 
\ointctrclockwise_{\partial D(z_0; \varepsilon)} d\zeta \, 
\big[(H_0 - \zeta I_{\cH}) (H - \zeta I_{\cH})^{-1}\big] (H_0 - \zeta I_{\cH})^{-1}\bigg)     \no \\
& \quad = - m_a(z_0; H_0)   \no \\
& \qquad - \f{1}{2 \pi i} {\tr}_{\cH}\bigg( 
\ointctrclockwise_{\partial D(z_0; \varepsilon)} d\zeta \, 
{\rm pp}_{z_0} \big\{\big[(H_0 - \zeta I_{\cH}) (H - \zeta I_{\cH})^{-1}\big] (H_0 - \zeta I_{\cH})^{-1}\big\}\bigg)     \no \\
& \quad = - m_a(z_0; H_0)   \no \\ 
& \qquad - \f{1}{2 \pi i}  
\ointctrclockwise_{\partial D(z_0; \varepsilon)} d\zeta \, 
 {\tr}_{\cH}\Big({\rm pp}_{z_0} \big\{\big[(H_0 - \zeta I_{\cH}) (H - \zeta I_{\cH})^{-1}\big] (H_0 - \zeta I_{\cH})^{-1}\big\}\Big)  \no \\
 & \quad = - m_a(z_0; H_0)   \no \\  
 & \qquad - \f{1}{2 \pi i}  
\ointctrclockwise_{\partial D(z_0; \varepsilon)} d\zeta \, 
 {\tr}_{\cH}\Big({\rm pp}_{z_0} \big\{(H_0 - \zeta I_{\cH})^{-1}\big[(H_0 - \zeta I_{\cH}) (H - \zeta I_{\cH})^{-1}\big]\big\}\Big)  \no \\
& \quad = - m_a(z_0; H_0) - \f{1}{2 \pi i} {\tr}_{\cH}\bigg( 
\ointctrclockwise_{\partial D(z_0; \varepsilon)} d\zeta \, 
 {\rm pp}_{z_0} \big\{(H - \zeta I_{\cH})^{-1}\big\}\bigg)  \no \\
& \quad = - m_a(z_0; H_0) - \f{1}{2 \pi i} {\tr}_{\cH} \bigg(\ointctrclockwise_{\partial D(z_0; \varepsilon)} d\zeta \, 
(H - \zeta I_{\cH})^{-1}\bigg)  \no \\
& \quad = - m_a(z_0; H_0) + {\tr}_{\cH} (P(z_0;H))     \no \\ 
& \quad = m_a(z_0;H) - m_a(z_0; H_0).       \lb{4.43}
\end{align}
Here we used that $z_0 \in \sigma_d(H_0)$ to arrive 
at the 3rd equality sign in \eqref{4.43}, employed again \eqref{4.29}, the fact that 
\begin{equation}
z\mapsto (H_0 - z I_{\cH}) (H - z I_{\cH})^{-1}=\big[I_{\cH} - K(z)\big]^{-1}
\end{equation}
is finitely 
meromorphic in $D(z_0;\varepsilon_0)$, applied \eqref{4.7}, and finally used \eqref{4.41}. 
\end{proof}
%%%%%%

%%%%%%
\begin{remark} \lb{r5.8}
Condition \eqref{4.33C}, viz., $[I_{\cH} - K_0(z_0)] \in \Phi(\cH)$, might not be so easily verified in practice. 
However, invoking compactness of $K(\dott)$ in the form that 
\begin{equation}
K(z) \in \cB_{\infty}(\cH), \quad 0 < |z - z_0| < \varepsilon_0, 
\end{equation}
readily implies that $K_0(z_0) \in \cB_{\infty}(\cH)$ and hence $[I_{\cH} - K_0(z_0)] \in \Phi(\cH)$ with vanishing Fredholm index, $\ind(I_{\cH} - K_0(z_0)) = 0$. Indeed, by \eqref{4.32B}, 
\begin{equation}
K(z) = \sum_{k=-N_0}^{\infty} (z-z_0)^k K_k(z_0) \in \cB_{\infty}(\cH), \quad 0 < |z - z_0| < \varepsilon_0, 
\end{equation}
and hence also 
\begin{equation}
\sum_{k=0}^{\infty} (z-z_0)^k K_k(z_0) \in \cB_{\infty}(\cH), \quad 0 < |z - z_0| < \varepsilon_0,  \lb{4.44} 
\end{equation}
as $K_k(z_0) \in \cF(\cH)$ for $- N_0 \leq k \leq -1$. Taking the norm limit $z \to z_0$ in \eqref{4.44} results in $K_0(z_0) \in \cB_{\infty}(\cH)$. \hfill $\diamond$
\end{remark} 
%%%%%%

We end this section with a variant of Theorem~\ref{t5.7}, where it is assumed that $z_0\in\rho(H_0)$. In this case $K(\dott)$ is analytic in $z_0$ and assumption \eqref{4.33C} simplifies: 

%%%%%%
\begin{theorem} \lb{t5.9}
Assume Hypothesis~\ref{h5.6} and consider a point $z_0 \in \rho(H_0) \cap \sigma(H)$ such that
$D(z_0; \varepsilon_0)\backslash \{z_0\}\subset \rho(H_0)\cap\rho(H)$ for some $\varepsilon_0 > 0$. In addition, 
suppose that the Birman--Schwinger operator satisfies
\begin{equation}
[I_{\cH} - K(z_0)] \in \Phi(\cH).  \lb{4.32A} 
\end{equation} 
Then $z_0$ is a discrete eigenvalue of $H$,  
\begin{equation}
z_0 \in \sigma_d(H),    \lb{4.39A} 
\end{equation}
and for $0 < \varepsilon<\varepsilon_0$, 
\begin{equation}
m_a(z_0; H) = m_a(z_0;I_{\cH} - K(\dott)) = {\ind}_{\partial D(z_0; \varepsilon)}(I_{\cH} - K(\dott)).   \lb{4.40A} 
\end{equation} 
\end{theorem}
%%%%%%

The proof of Theorem~\ref{t5.9} is almost the same as the proof of Theorem~\ref{t5.7};
instead of the meromorphic Fredholm theorem (Theorem~\ref{tA.4}) it now suffices to use the analytic Fredholm theorem (Theorem~\ref{tA.2}). 
Finally use Theorem~\ref{t5.4}.

%%%%%%
%%%%%%
\section{Essential Spectra and the Weinstein--Aronszajn Formula}\lb{s6}
%%%%%%
%%%%%%

In our final section we discuss various issues connected with essential spectra of closed operators.

Throughout this section we assume that $A$ is a closed operator in the separable, complex Hilbert space $\cK$. 
Recalling the definition of the discrete spectrum $\sigma_d(A)$ of $A$ in \eqref{2.7}, we now
introduce (in analogy to the self-adjoint case) the essential spectrum $\wti \sigma_{ess}(A)$ of $A$ as follows:

%%%%%%%
\begin{definition} \lb{d6.1}
Let $A$ be a closed operator in $\cK$. Then the {\it essential spectrum of $A$} is defined by
\begin{equation}
\wti \sigma_{ess}(A) = \sigma(A) \backslash \sigma_d(A)    \lb{6.3} 
\end{equation}
\end{definition}
%%%%%%%

One verifies that 
\begin{equation}
\sigma(A), \; \wti \sigma_{ess}(A) \, \text{ are closed subsets of $\bbC$,}    \lb{6.5} 
\end{equation}
and (cf.\ Definition~\ref{d5.2})
\begin{equation}
\wti \rho(A) = \bbC \backslash \wti \sigma_{ess}(A).    \lb{6.6} 
\end{equation}

Of course, $\wti \sigma_{ess}(A)$ coincides with the standard essential spectrum of $A$ if $A$ is self-adjoint in $\cK$. Since $A$ is not assumed to be self-adjoint, we emphasize that several inequivalent definitions of the essential spectrum are in use in the literature (cf., e.g., \cite[Sects.~11.2, 14.4]{Da07}, and especially, \cite[Ch.~9]{EE18} for a detailed discussion), 
however, in this paper we will only adhere to \eqref{6.3}, which corresponds to the definition employed in \cite[p.~106]{RS78}, and, as discussed in detail in \cite[p.~30--56]{BC19} (especially, Lemma~III.125 on p.~53) and in \cite[App.~B]{HL07}, our definition of $\wti \sigma_{ess}(A)$ in \eqref{6.3} is precisely $\sigma_{e5}(A)$ (also known as the Browder essential spectrum) in \cite[Sect.~1.4, Ch.~9]{EE18} (see, especially, pp.~460--461). 

Next, we turn to the invariance of the essential spectrum with respect to relatively compact perturbations. 
We start by recalling the following well-known facts. Suppose  that $A$ is closed in $\cK$, $\rho(A) \neq \emptyset$, and $B$ is an operator in $\cK$ with $\dom(B) \supseteq \dom(A)$. Then $B$ is {\it relatively bounded} (resp., {\it relatively compact}) with respect to $A$ if and only if $B(A - z_0 I_{\cK})^{-1} \in \cB(\cK)$ (resp., 
$B(A - z_0 I_{\cK})^{-1} \in \cB_{\infty}(\cK)$) for some (and hence for all) $z_0 \in \rho(A)$. In particular, if $B$ is relatively compact with respect to $A$ then it is infinitesimally bounded with respect to $A$ and consequently, the operator 
\begin{equation}
A+B, \, \text{ defined on } \, \dom(A+B) = \dom(A), \, \text{ is closed in $\cK$,} 
\end{equation} 
a fact that will be tacitly employed in the Theorem~\ref{t6.3}. (For details, see, e.g., \cite[Sects.~IV.1.1, IV.1.3]{Ka80}, 
\cite[Sect.~ 9.2]{We00}.) 
The following theorem on the extended resolvent set and the essential spectrum is in a certain sense folklore and is 
widely used by experts on non-self-adjoint operators. It can be seen as a variant of \cite[Lemma~I.5.2]{GK69} for unbounded closed operators
and can also be concluded from stability theorems for semi-Fredholm operators; cf. \cite[Sect.~IV.5]{Ka80} and Remark~\ref{r6.5}~(iii).

%%%%%%
\begin{theorem} \lb{t6.3}
Assume that $A$ is closed in $\cK$, $\rho(A) \neq \emptyset$, and $B$ is an operator in $\cK$ with 
$\dom(B) \supseteq \dom(A)$ and 
\begin{equation} 
B(A - z I_{\cK})^{-1} \in \cB_{\infty}(\cK) \, \text{ for some $($and hence for all\,$)$ $z \in \rho(A)$.} 
\end{equation} 
$(i)$ Let $\Omega_0(A)$ be a connected component of $\wti \rho(A)$. If 
$\Omega_0(A) \cap \rho(A+B) \neq \emptyset$, then $\Omega_0(A)$ is also a connected component of 
$\wti \rho(A + B)$. In particular, if $A, B \in \cB(\cK)$, the unbounded connected component of $\wti \rho(A)$ and 
$\wti \rho(A + B)$ coincide. \\[1mm]
$(ii)$ Suppose that $\Omega_j(A) \cap \rho(A+B) \neq \emptyset$ for all connected components $\Omega_j(A)$, 
$j \in J$, of $\wti \rho(A)$ $($$J \subseteq \bbN$ an appropriate index set\,$)$. Then,
\begin{equation}
\wti \rho(A) \subseteq \wti \rho(A+B), \, \text{ equivalently, } \, \wti \sigma_{ess}(A+B) \subseteq \wti \sigma_{ess}(A).
\end{equation}
\end{theorem}
%%%%%%
\begin{proof}
$(i)$ One can follow the proof of \cite[Lemma~I.5.2]{GK69}, where the special case of bounded operators is treated. More precisely, 
one first notes that
$K(z) = - B(A - z I_{\cK})^{-1}$ is analytic in $\Omega_0(A) \cap \rho(A)$ and  
$K(z) \in \cB_{\infty}(\cK)$ for all $z \in \Omega_0(A) \cap \rho(A)$. By hypothesis, $(A+B-z_0 I_{\cK})^{-1} \in \cB(\cK)$ for some $z_0 \in \Omega_0(A)$. If $z_0 \in \rho(A)$, then 
\begin{equation} \label{e6.7}
(A+B-z_0 I_{\cK}) = [I_{\cK} - K(z_0)] (A-z_0 I_{\cK})
\end{equation}
implies bounded invertibility of $I_{\cK} - K(z_0)$, that is, 
\begin{equation}
 [I_{\cK} - K(z_0)]^{-1} \in \cB(\cK).
\end{equation}
If $z_0 \in \sigma_d(A)$ one chooses a $z_1 \neq z_0$, $z_1 \in \rho(A) \cap \Omega_0(A)$ in a sufficiently small neighborhood of $z_0$ and then obtains
\begin{equation}
 [I_{\cK} - K(z_1)]^{-1} \in \cB(\cK).
\end{equation}
This choice is possible since $\rho(A), \Omega_0(A)$ are both open subsets of $\bbC$ and $z_0 \in \sigma_d(A)$ is isolated in $\sigma(A)$ by definition. By Theorem~\ref{tA.2}, 
there is a discrete set $\cD_0 \subset \rho(A) \cap \Omega_0(A)$ such that
$[I_{\cK} - K(z)]^{-1} \in \cB(\cK)$ for all
$z \in \rho(A) \cap \Omega_0(A)$ with $z \not\in \cD_0$. Thus, since $\sigma(A) \cap \Omega_0(A)$ consists of isolated points only, 
\begin{equation}
(A+B-z I_{\cK}) = [I_{\cK} - K(z)] (A-z I_{\cK})
\end{equation}
is boundedly invertible for $z \in \Omega_0(A) \backslash \cD_1$, where
$\cD_1 \subset \Omega_0(A)$ is a (possibly empty) discrete set. In particular, $\sigma(A+B) \cap \Omega_0(A)$ consists of isolated points only. 
Next, pick $\wti z \in \Omega_0(A)$ and $0 < \varepsilon$ sufficiently small such that 
$\ol{D(\wti z;\varepsilon)} \backslash \{\wti z\} \subset \rho(A) \cap \rho(A+B)$.
Use \eqref{e6.7} and consider the Riesz projection
\begin{align}
& P(\wti z; A+B) = \f{-1}{2 \pi i} \ointctrclockwise_{\partial D(\wti z;\varepsilon)} d \zeta \, (A+B - \zeta I_{\cK})^{-1} 
\no \\
& \quad = \f{-1}{2 \pi i} \ointctrclockwise_{\partial D(\wti z;\varepsilon)} d \zeta \, 
\Big((A - \zeta I_{\cK})^{-1} - (A+B - \zeta I_{\cK})^{-1} \big[B(A - \zeta I_{\cK})^{-1}\big]\Big)   \no \\
& \quad = P(\wti z; A) - \f{1}{2 \pi i} \ointctrclockwise_{\partial D(\wti z;\varepsilon)} d \zeta \, 
(A+B - \zeta I_{\cK})^{-1} K(\zeta). 
\end{align}
Since $\dim(\ran(P(\wti z; A))) < \infty$ and $K(\zeta) \in \cB_{\infty}(\cK)$ for all $\zeta\in\partial D(\wti z;\varepsilon)$, also 
$P(\wti z; A+B) \in \cB_{\infty}(\cK)$. The latter fact is equivalent to $\dim(\ran(P(\wti z; A+B))) < \infty$ implying 
$\wti z \in \wti \rho(A+B)$ and hence 
\begin{equation}
\Omega_0(A) \subseteq \wti \rho(A+B).
\end{equation} 
In particular,
\begin{equation}
\Omega_0(A) \subseteq \Omega_0 (A+B),     \lb{6.13a} 
\end{equation} 
where $\Omega_0 (A+B)$ is the connected component of $\wti \rho(A+B)$ that contains $\Omega_0(A)$. To prove the reverse inclusion we now interchange the role of $A$ and $A+B$ as follows: Pick $\hatt z \in \wti \rho(A+B)$ 
and $0 < \varepsilon$ sufficiently small such that 
$\ol{D(\hatt z;\varepsilon)} \backslash \{\hatt z\} \subset \rho(A) \cap \rho(A+B)$ and consider the Riesz projection
\begin{align}
& P(\hatt z; A) = \f{-1}{2 \pi i} \ointctrclockwise_{\partial D(\hatt z;\varepsilon)} d \zeta \, (A - \zeta I_{\cK})^{-1} 
\no \\
& \quad = \f{-1}{2 \pi i} \ointctrclockwise_{\partial D(\hatt z;\varepsilon)} d \zeta \, 
\Big((A + B - \zeta I_{\cK})^{-1} + (A+B - \zeta I_{\cK})^{-1} \big[B(A - \zeta I_{\cK})^{-1}\big]\Big)   \no \\
& \quad = P(\hatt z; A+B) + \f{1}{2 \pi i} \ointctrclockwise_{\partial D(\hatt z;\varepsilon)} d \zeta \, 
(A+B - \zeta I_{\cK})^{-1} K(\zeta). 
\end{align}
Invoking compactness of $K(\dott)$ once more one then concludes $\hatt z \in \wti \rho(A)$ and hence 
\begin{equation}
\Omega_0(A + B) \subseteq \wti \rho(A).
\end{equation} 
Since $\Omega_0(A + B) \cap \Omega_0(A) \neq \emptyset$, this implies 
\begin{equation}
\Omega_0(A + B) \subseteq \Omega_0 (A),      \lb{6.16a}
\end{equation} 
and hence by \eqref{6.13a},  
\begin{equation}
\Omega_0 (A) = \Omega_0 (A+B).
\end{equation}
The fact that for bounded operators
\begin{equation}
\rho(A) \cap \rho(A+B) \supset \{z \in \bbC \, | \, |z| > \max(\|A\|, \|A+B\|)\},
\end{equation}
proves that the unbounded connected component of $\rho(A)$ and $\rho(A+B)$ coincide. \\[1mm]
$(ii)$ Since $\bigcup_{j \in J} \Omega_j(A) = \wti \rho(A)$, part $(ii)$ follows from part $(i)$. 
\end{proof}
%%%%%%

Thus, a sufficient condition for the invariance of the essential spectrum of $A$ under the perturbation $B$ 
arises as follows:

%%%%%%
\begin{corollary} \lb{c6.4}
Assume that $A$ is closed in $\cK$, $\rho(A) \neq \emptyset$, and $B$ is a linear operator in $\cK$ with 
$\dom(B) \supseteq \dom(A)$ and 
\begin{equation}\label{usethis234}
B(A - z I_{\cK})^{-1} \in \cB_{\infty}(\cK) \, \text{ for some $($and hence for all\,$)$ $z \in \rho(A)$.}
\end{equation} 
Suppose each connected component of $\wti \rho(A)$ contains a point of $\rho(A+B)$, and each connected component of $\wti \rho(A+B)$ contains a point of $\rho(A)$. Then,
\begin{equation}
\wti \rho(A+B) = \wti \rho(A), \, \text{ equivalently, } \, \wti \sigma_{ess}(A+B) = \wti \sigma_{ess}(A).
\end{equation}
\end{corollary}
%%%%%%
\begin{proof} 
From Theorem~\ref{t6.3}\,$(ii)$ we obtain the inclusion
\begin{equation}
\wti \rho(A) \subseteq \wti \rho(A+B), \, \text{ equivalently, } \, \wti \sigma_{ess}(A+B) \subseteq \wti \sigma_{ess}(A).
\end{equation}
To show the other inclusion we note that $A+B$ is closed in $\cK$, $\rho(A+B)\neq \emptyset$, and $-B$ is a linear operator in $\cK$ with $\dom(-B) \supseteq \dom(A+B)$. Moreover, for $z\in\rho(A)\cap\rho(A+B)$ one verifies  
\begin{equation}
B(A+B - z I_{\cK})^{-1}=I_\cK-(A-zI_{\cK}) (A+B - z I_{\cK})^{-1}\in\cB(\cK)
\end{equation} 
and from
\begin{equation}
- B (A+B - z I_{\cK})^{-1} = - B(A - z I_{\cK})^{-1} + B(A - z I_{\cK})^{-1} \big[B(A+B - z I_{\cK})^{-1}\big]
\end{equation}
and assumption \eqref{usethis234} one concludes $- B (A+B - z I_{\cK})^{-1} \in \cB_{\infty}(\cK)$ for some 
$($and hence for all\,$)$ $z \in \rho(A)\cap\rho(A+B)$. Therefore, Theorem~\ref{t6.3}\,$(ii)$ applies with the roles of $A$ and $A+B$ interchanged to obtain the remaining 
inclusion
\begin{equation}
\wti \rho(A+B) \subseteq \wti \rho(A), \, \text{ equivalently, } \, \wti \sigma_{ess}(A) \subseteq \wti \sigma_{ess}(A+B).
\end{equation}
\end{proof}
%%%%%%

%%%%%%
\begin{remark} \lb{r6.5}
$(i)$ The example of the unitary left-shift operator $A_0$ in $\ell^2(\bbZ)$, given by 
\begin{equation}
(A_0 f)_n = f_{n+1}, \quad f = \{f_n\}_{n \in \bbZ} \in \ell^2(\bbZ), 
\end{equation}
with 
\begin{equation}
\sigma(A_0) = \sigma_{ess}(A_0) = \{z \in \bbC \, | \, |z| =1\} = \partial D(0; 1),
\end{equation}
perturbed by the rank-one perturbation
\begin{equation}
(B_0 f)_n = - \delta_{n,0} f_1, \quad f = \{f_n\}_{n \in \bbZ} \in \ell^2(\bbZ), 
\end{equation}
yields for $A_0 + B_0$,
\begin{equation}
\sigma(A_0 + B_0) = \sigma_{ess}(A_0 + B_0) = \ol{D(0;1)},
\end{equation}
see, for instance, \cite[Example~1, p.~110]{RS78}. Thus, invariance of the essential spectrum already fails 
spectacularly even in the presence of a rank-one perturbation $B_0$, without some some additional 
condition on the connected components of $\wti \rho(A_0)$, respectively, $\wti \rho(A_0 + B_0)$. Indeed, 
\begin{align}
& \rho(A_0) = \wti \rho(A_0) = D(0;1) \cup (\bbC \backslash \ol{D(0;1)}) = \bbC \backslash \partial D(0;1)  
\end{align} 
consists of two connected components, while 
\begin{align}
& \rho(A_0 + B_0) = \wti \rho(A_0+ B_0) = \bbC \backslash \ol{D(0;1)} 
\end{align} 
has precisely one connected component, consistent with 
\begin{equation} 
\ol{D(0;1)} = \sigma_{ess}(A_0 + B_0) \supsetneqq \sigma_{ess}(A_0) = \partial D(0; 1).
\end{equation} 
$(ii)$ For additional criteria implying invariance of the essential spectrum we refer to \cite[p.~111--117]{RS78}. 
\\[1mm]
$(iii)$ Alternatively, one can prove Theorem~\ref{t6.3} invoking Fredholm theoretic notions (such as, nullity, 
deficiency, and the Fredholm index). For more details we refer to \cite[Sect.~IV.5]{Ka80}.  \\[1mm] 
$(iv)$ The hypotheses on the connected components of $\widetilde \rho(A)$ and $\widetilde \rho(A+B)$ in 
Corollary~\ref{c6.4} correct and complete the ones made in \cite[Sects.~4, 5]{GLMZ05}.
\hfill $\diamond$ 
\end{remark}
%%%%%%

Next, we need one more piece of notation: Let
$\Omega\subseteq\bbC$ be open and connected, and let
$f\colon\Omega\to\bbC\cup\{\infty\}$ be meromorphic and not
identically vanishing on $\Omega$. The {\it multiplicity function}
$m(z;f)$, $z\in\Omega$, is then defined by
\begin{align}
m(z;f)&=\begin{cases} k, & \text{if $z$ is a zero of $f$ of order
$k$,} \\
-k, & \text{if $z$ is a pole of $f$ of order $k$,} \\
0, & \text{otherwise.} \end{cases} \lb{6.34} \\
&= \f{1}{2\pi i} \ointctrclockwise_{\partial D(z; \varepsilon) } d\zeta \,
\f{f'(\zeta)}{f(\zeta)}, \quad z\in\Omega, \lb{6.35}
\end{align}
for $0 < \varepsilon$ sufficiently small. Here the counterclockwise oriented circle $\partial D(z; \varepsilon) $ 
is chosen sufficiently small such that $\partial D(z; \varepsilon) $ contains no other singularities or zeros of $f$ except, possibly, $z$.

As discussed in Howland \cite{Ho70}, there is an additional problem with 
meromorphic (even finitely meromorphic) $\cB_p(\cH)$-valued
functions in connection with modified Fredholm determinants we need to address. 
Indeed, suppose $F$ is meromorphic in $\Omega$ and $F(z)$, $z \in \Omega$, is of finite rank. Then
of course $\det_\cH(I_{\cH}-F(\cdot))$ is meromorphic in $\Omega$ (here $\det_\cH(\dott)$ represents the standard 
{\it Fredholm determinant}). However, the formula (see, e.g., \cite[p.~75]{Si05}, \cite[p.~44]{Ya92}), 
\begin{equation}
{\det}_{\cH,p} (I_{\cH}-F(z))=  {\det}_{\cH}(I_{\cH}-F(z))
\exp\bigg[\tr_\cH\bigg(-\sum_{j=1}^{p-1}j^{-1} F(z)^j\bigg)\bigg],
\quad z\in\Omega,  \lb{6.37}
\end{equation}
where ${\det}_{\cH,p} (\dott)$, $p \in \bbN$, represents the $p$th ({\it modified}) {\it Fredholm determinant} 
(cf., \cite[Ch.~IX]{GGK00}, \cite[Ch.~IV]{GK69}, \cite{Si77}, \cite[Chs.~3, 9]{Si05}, \cite[Sect.~1.7]{Ya92}) shows that ${\det}_{\cH,p} (I_{\cH}-F(\dott))$, for $p > 1$, in general, will exhibit essential singularities at poles 
of $F$. To sidestep this difficulty, Howland extends the definition of $m(\cdot\,;f)$ in \eqref{6.34} to functions $f$ with isolated essential singularities by focusing exclusively on \eqref{6.35}: Suppose $f$ is meromorphic in
$\Omega$ except at isolated essential singularities. Then as in \eqref{6.35}, we use the definition 
\begin{equation}
m(z;f)=\f{1}{2\pi i} \ointctrclockwise_{\partial D(z; \varepsilon) } d\zeta \,
\f{f'(\zeta)}{f(\zeta)}, \quad z\in\Omega, \lb{6.38}
\end{equation}
where $\varepsilon>0$ is again chosen sufficiently small to exclude all singularities and zeros of $f$ except, 
possibly, $z$.

Given the extension of $m(\dott\,;f)$ to meromorphic functions with isolated essential singularities, the 
generalization of Howland's global Weinstein--Aronszajn formula \cite{Ho70} to the case where 
$H_0, H$ are non-self-adjoint, is obtained in the next theorem. Here a slightly stronger assumption than in Theorem~\ref{t5.7} or Remark~\ref{r5.8} is imposed,
namely, it is assumed that the Birman--Schwinger operator satisfies $K(z)=-V(H_0 - z I_{\cH})^{-1}\in \cB_p(\cH)$ for some $p \in \bbN$
and for some (and hence for all) $z \in \rho(H_0)$.
In this situation it will be shown in the proof of Theorem~\ref{t6.6} that the identity
\begin{equation}\label{puthere}
 m(z_0; {\det}_{\cH,p} (I_{\cH} - K(\dott)))={\ind}_{\partial D(z_0; \varepsilon)}(I_{\cH} - K(\dott)) 
\end{equation}
holds for all $z_0 \in \wti \rho (H_0)$ and $0 < \varepsilon$ sufficiently small.

%%%%%%%
\begin{theorem} \lb{t6.6}
In addition to Hypothesis~\ref{h5.6} let $p \in \bbN$, and assume that 
\begin{equation} 
V(H_0 - z I_{\cH})^{-1} \in \cB_p(\cH) \, \text{ for some $($and hence for all\,$)$ $z \in \rho(H_0)$.}
\end{equation} 
Suppose each connected component of $\wti \rho(H_0)$ contains a point of $\rho(H)$, and each connected component of $\wti \rho(H)$ contains a point of $\rho(H_0)$. Then,
\begin{equation}
\wti \rho(H) = \wti \rho(H_0), \, \text{ equivalently, } \, \wti \sigma_{ess}(H) = \wti \sigma_{ess}(H_0),  \lb{6.39} 
\end{equation}
and the global Weinstein--Aronszajn formula
\begin{equation}
m_a(z;H)=m_a(z;H_0)+m(z;{\det}_{\cH,p}(I_{\cK}-K(\dott))), \quad z \in \wti \rho (H_0),      \lb{6.40}
\end{equation}
holds. 
\end{theorem}
%%%%%%%
\begin{proof}
Since \eqref{6.39} has been established in Corollary~\ref{c6.4}, it suffices to focus on \eqref{6.40}. 
In the computation \eqref{6.41} below we shall make use of the elementary identity 
\begin{equation}  
[I_{\cK} - L]^{-1}L^{p-1}=[I_{\cK} - L]^{-1} -  \sum_{j=0}^{p-2} L^j 
\end{equation} 
for $L \in \cB(\cK)$ such that $[I_{\cK} - L]^{-1} \in \cB(\cK)$.
It follows from \eqref{6.38}, \cite[eq.~(18) on p.~44]{Ya92}
for $z_0 \in \wti \rho (H_0)$, and $0 < \varepsilon$ sufficiently small, that 
\begin{align}
& m(z_0; {\det}_{\cH,p} (I_{\cH} - K(\dott))) = \f{1}{2 \pi i}  \ointctrclockwise_{\partial D(z_0; \varepsilon)} d \zeta \, 
\f{d}{d \zeta} \ln({\det}_{\cH,p}(I_{\cH} - K(\zeta)))    \no \\
& \quad = \f{1}{2 \pi i}  \ointctrclockwise_{\partial D(z_0; \varepsilon)} d \zeta \, 
\tr_{\cH} \big([I_{\cH} - K(\zeta)]^{-1} K(\zeta)^{p-1} [-K'(\zeta)]\big)   \no \\
& \quad = \f{1}{2 \pi i} \tr_{\cH} \bigg( \ointctrclockwise_{\partial D(z_0; \varepsilon)} d \zeta \, 
[I_{\cH} - K(\zeta)]^{-1} K(\zeta)^{p-1} [- K'(\zeta)]\bigg)   \no \\ 
& \quad = \f{1}{2 \pi i} \tr_{\cH} \bigg( \ointctrclockwise_{\partial D(z_0; \varepsilon)} d \zeta \,  
\bigg\{\sum_{j=0}^{p-2} K(\zeta)^j K'(\zeta) + [I_{\cH} - K(\zeta)]^{-1} [- K'(\zeta)] \bigg\}\bigg)   \no \\
& \quad = \f{1}{2 \pi i} \tr_{\cH} \bigg( \ointctrclockwise_{\partial D(z_0; \varepsilon)} d \zeta \,  
[I_{\cH} - K(\zeta)]^{-1} [- K'(\zeta)] \bigg)   \no \\
& \quad = {\ind}_{\partial D(z_0; \varepsilon)}(I_{\cH} - K(\dott))    \no \\
& \quad = m_a(z_0; H) - m_a(z_0; H_0),      \lb{6.41} 
\end{align}
which also implies \eqref{puthere}.
Here we used \eqref{4.8} and \eqref{4.40} in the final steps, and the fact that because of analyticity of 
$K(\dott)^j K'(\dott)$ in a sufficiently small punctured neighborhood of $z_0$,  
\begin{equation}
\ointctrclockwise_{\partial D(z_0; \varepsilon)} d \zeta \,  
K(\zeta)^j K'(\zeta) = 0, \quad j \in \bbN_0.   \lb{6.42}
\end{equation} 
To prove \eqref{6.42} one invokes \eqref{ws}--\eqref{4.7} repeatedly to conclude that
\begin{align}
& \tr_{\cH} \bigg( \ointctrclockwise_{\partial D(z_0; \varepsilon)} d \zeta \,  
K(\zeta)^j K'(\zeta)\bigg) = \f{1}{j+1} \tr_{\cH} \bigg( \ointctrclockwise_{\partial D(z_0; \varepsilon)} d \zeta \,  
\f{d}{d \zeta} K(\zeta)^{j+1}\bigg) = 0,      \no \\
& \hspace*{9.5cm} j \in \bbN_0,  
\end{align}
since $K(\dott)^{j+1}$, $j \in \bbN_0$, is analytic in a sufficiently small neighborhood of 
$\partial D(z_0; \varepsilon)$. The result \eqref{6.41} extends to $z_0$ in each connected component 
of $\wti \rho (H_0)$ and hence to all of $\wti \rho (H_0)$. 
\end{proof}
%%%%%%%

%%%%%%%
\begin{remark} \lb{r6.7}
In the special case $p=1$ Theorem~\ref{t6.6} 
was originally obtained by Kuroda \cite{Ku61}. Howland \cite{Ho70} developed the theory in great detail for 
very general perturbations (patterned after Kato \cite{Ka66}) in the case where $H_0$ and $H$ are self-adjoint. 
A very different proof of \eqref{6.40}, closely following the arguments in Howland \cite{Ho70}, was discussed in \cite{GLMZ05}. Frank \cite{Fr18} also considered the case where $H_0$ 
was self-adjoint and bounded from below. \hfill $\diamond$
\end{remark}
%%%%%%%

We conclude with the following variant of Theorem~\ref{t5.9}:  

%%%%%%
\begin{theorem} \lb{t6.8}
Assume Hypothesis~\ref{h5.6} and consider a point $z_0 \in \rho(H_0) \cap \sigma(H)$ such that
$D(z_0; \varepsilon_0)\backslash \{z_0\}\subset \rho(H_0)\cap\rho(H)$ for some $\varepsilon_0 > 0$. Let $p \in \bbN$. In addition, 
suppose that 
\begin{equation} 
V(H_0 - z I_{\cH})^{-1} \in \cB_p(\cH) \, \text{ for some $($and hence for all\,$)$ $z \in \rho(H_0)$.}   \lb{6.43} 
\end{equation} 
Then $z_0$ is a discrete eigenvalue of $H$,  
\begin{equation}
z_0 \in \sigma_d(H),    \lb{6.44} 
\end{equation}
and for $0 < \varepsilon<\varepsilon_0$, 
\begin{align}
\begin{split} 
m_a(z_0;H) &= {\ind}_{\partial D(z_0; \varepsilon)}(I_{\cH} - K(\dott))    \\     
&= m_a(z_0; I_{\cH} - K(\dott)) = m(z_0;{\det}_{\cH,p}(I_{\cK}-K(\dott))).    \lb{6.45}
\end{split} 
\end{align}
\end{theorem}
%%%%%%
\begin{proof}
By formula \eqref{4.40A} in Theorem~\ref{t5.9} (which applies in the current situation since assumption 
\eqref{6.43} is stronger than \eqref{4.32A}; cf. Remark~\ref{r5.8}), it remains to prove the final equality in \eqref{6.45}. For this 
purpose one now follows the derivation of \eqref{6.41} (see also \eqref{puthere}) line by line to arrive at 
\begin{equation}
m(z_0; {\det}_{\cH,p} (I_{\cH} - K(\dott))) = \cdots\cdots  
= {\ind}_{\partial D(z_0; \varepsilon)}(I_{\cH} - K(\dott)) = m_a(z_0; H),  
\end{equation}
since $m_a(z_0; H_0) = 0$ as $z_0 \in \rho(H_0)$. 
\end{proof}
%%%%%%

We remark that
Lemma~3.2 in Frank \cite{Fr18} corresponds to the second line of \eqref{6.45} (see also \cite{GK17}). 

\appendix

%%%%%%% 
%%%%%%% 
\section{The Analytic and Meromorphic Fredholm Theorems} \lb{sA} 
%%%%%%% 
%%%%%%% 
In this short appendix we recall the analytic and meromorphic Fredholm theorems for operator-valued functions, see, e.g., 
\cite[Sect.~4.1]{GL09}, \cite{GS71}, \cite{Ho70}, \cite[Theorem\ VI.14]{RS80}, \cite[Theorem~XIII.13]{RS78}, \cite{RV69}, \cite{St68}.

%%%%%%%%%
\begin{hypothesis} \lb{hA.1} 
Let $\Omega \subseteq \bbC$ be open and connected, and suppose that 
$A:\Omega \to \cB(\cH)$ is analytic and that 
\begin{equation}
A(z) \in \Phi(\cH) \, \text{ for all } \, z\in\Omega.    \lb{A.1} 
\end{equation}  
\end{hypothesis}
%%%%%%%%% 

Then the analytic Fredholm theorem reads as follows:

%%%%%%%%%
\begin{theorem} \lb{tA.2} 
Assume that $A:\Omega \to \cB(\cH)$ satisfies Hypothesis~\ref{hA.1}. Then either \\
$(i)$ $A(z)$ is not boundedly invertible for any $z \in \Omega$, \\[1mm]
or else, \\[1mm]
$(ii)$ $A(\dott)^{-1}$ is finitely meromorphic on $\Omega$. More precisely, there exists a discrete 
subset $\Omega_d \subset \Omega$ $($i.e., a set without limit points in $\Omega$;  
possibly, $\Omega_d = \emptyset$$)$ such that 
$A(z)^{-1} \in \cB(\cH)$ for all $z \in \Omega\backslash \Omega_d$, $A(\dott)^{-1}$ is analytic on 
$\Omega\backslash \Omega_d$, and meromorphic on $\Omega$. In addition, 
\begin{equation}
A(z)^{-1} \in \Phi(\cH) \, \text{ for all } \, z \in \Omega\backslash \Omega_d, 
\end{equation}
and if $z_1 \in \Omega_d$ then 
\begin{equation}
A(z)^{-1} = \sum_{k= - N_0(z_1)}^{\infty} (z - z_1)^{k} C_k(z_1), \quad 
0 < |z - z_1| < \varepsilon_0(z_1),    \lb{A.3}
\end{equation}
with  
\begin{align} \lb{A.4} 
\begin{split} 
& C_{-k}(z_1) \in \cF(\cH), \; 1 \leq k \leq N_0(z_1),  \quad  C_0(z_1) \in \Phi(\cH), \\
& C_k(z_1) \in \cB(\cH), \; k \in \bbN. 
\end{split} 
\end{align}
In addition, if $[I_{\cH} - A(z)] \in \cB_{\infty}(\cH)$ for all $z \in \Omega$, then 
\begin{equation}
\big[I_{\cH} - A(z)^{-1}\big] \in \cB_{\infty}(\cH), \; z \in \Omega\backslash \Omega_d, 
\quad [I_{\cH} - C_0(z_1)] \in \cB_{\infty}(\cH), \; z_1 \in \Omega_d.  
\end{equation} 
\end{theorem}
%%%%%%%%%

Next, we briefly turn to the meromorphic Fredholm theorem.

%%%%%%%%%
\begin{hypothesis} \lb{hA.3} 
Let $\Omega \subseteq \bbC$ be open and connected, and $\Omega_d \subset \Omega$ a discrete set 
$($i.e., a set without limit points in $\Omega$$)$. Suppose that $A:\Omega \backslash \Omega_d \to \cB(\cH)$ 
is analytic and that $A(\dott)$ is finitely meromorphic on $\Omega$. In addition, suppose that  
\begin{equation}
A(z) \in \Phi(\cH) \, \text{ for all } \, z\in\Omega \backslash \Omega_d,    \lb{A.6} 
\end{equation}  
and for all $z_0 \in \Omega_d$ there is a norm convergent Laurent expansion
around $z_0$ of the form
\begin{equation}
A(z) = \sum_{k=-N_0}^{\infty} (z-z_0)^k A_k(z_0), \quad 
0 < |z - z_0| < \varepsilon_0,    \lb{A.7}
\end{equation}
for some $N_0 = N_0(z_0) \in \bbN$ and some $0 < \varepsilon_0 = \varepsilon_0(z_0)$ 
sufficiently small, with  
\begin{align}\lb{A.8} 
\begin{split} 
& A_{-k}(z_0) \in \cF(\cH), \; 1 \leq k \leq N_0(z_0),  \quad  A_0(z_0) \in \Phi(\cH), \\
& A_k(z_0) \in \cB(\cH), \; k \in \bbN. 
\end{split} 
\end{align} 
\end{hypothesis}
%%%%%%%%% 

Then the meromorphic Fredholm theorem reads as follows:

%%%%%%%%%
\begin{theorem} \lb{tA.4} 
Assume that $A:\Omega\backslash \Omega_d \to \cB(\cH)$ satisfies Hypothesis~\ref{hA.3}. Then either \\
$(i)$ $A(z)$ is not boundedly invertible for any $z \in \Omega\backslash \Omega_d$, \\[1mm]
or else, \\[1mm]
$(ii)$ $A(\dott)^{-1}$ is finitely meromorphic on $\Omega$. More precisely, there exists a discrete 
subset $\Omega_{d,0} \subset \Omega$ $($possibly, $\Omega_{d,0} = \emptyset$$)$ such that 
$A(z)^{-1} \in \cB(\cH)$ for all $z \in \Omega\backslash\{\Omega_d \cup \Omega_{d,0}\}$, $A(\dott)^{-1}$ 
extends to an analytic function on $\Omega\backslash \Omega_{d,0}$, meromorphic on $\Omega$. In addition, 
\begin{equation}\lb{A.9}
A(z)^{-1} \in \Phi(\cH) \, \text{ for all } \, z \in \Omega\backslash \Omega_{d,0}, 
\end{equation}
and if $z_1 \in \Omega_{d,0}$ then for some $N_0(z_1) \in \bbN$, and for some $0 < \varepsilon_0(z_1)$ 
sufficiently small, 
\begin{equation}
A(z)^{-1} = \sum_{k= - N_0(z_1)}^{\infty} (z - z_1)^{k} D_k(z_1), \quad 
0 < |z - z_1| < \varepsilon_0(z_1),    \lb{A.10}
\end{equation}
with  
\begin{align} \lb{A.11} 
\begin{split} 
& D_{-k}(z_1) \in \cF(\cH), \;1 \leq k \leq N_0(z_1),  \quad  D_0(z_1) \in \Phi(\cH), \\
& D_k(z_1) \in \cB(\cH), \; k \in \bbN. 
\end{split} 
\end{align}
In addition, if $[I_{\cH} - A(z)] \in \cB_{\infty}(\cH)$ for all $z \in \Omega\backslash\Omega_d$, then 
\begin{equation}
\big[I_{\cH} - A(z)^{-1}\big] \in \cB_{\infty}(\cH), \quad z \in \Omega\backslash\Omega_{d,0}, 
\quad [I_{\cH} - D_0(z_1)] \in \cB_{\infty}(\cH), \quad z_1 \in \Omega_{d,0}.  
\end{equation} 
\end{theorem}
%%%%%%%%%

%%%%%%%
\medskip
\noindent {\bf Acknowledgments.} 
We are indebted to Rupert Frank, Yuri Latushkin, and Alim Sukhtayev for very helpful discussions. 
J.B.\ gratefully acknowledges support for the Distinguished Visiting Austrian Chair at Stanford University by the Europe Center and the
Freeman Spogli Institute for International Studies, where 
this work was completed in the spring of 2020. J.B.\ is also most grateful 
for the stimulating research
stay and the hospitality at the University of Auckland, where some
parts of this paper were written. This work is supported 
by the Marsden Fund Council from Government funding, administered by the Royal Society of New Zealand.
F.G.\ gratefully acknowledges kind invitations to the Institute for Applied Mathematics at the Graz University of 
Technology, Austria, for parts of June 2018 and June 2019. The extraordinary hospitality as well 
as the stimulating atmosphere at the Graz University of Technology is greatly appreciated. 
%%%%%%%

%%%%%%%
%%%%%%%

\end{document}